\documentclass[11pt]{article}
 \usepackage{amsmath, amsthm, amssymb, bbm, setspace,bigints}
 \usepackage[margin=1 in]{geometry}
\usepackage{caption}
\usepackage{subcaption}
\usepackage[toc,page]{appendix}
\usepackage{cite}         
\usepackage[pdftex]{graphicx}
\usepackage{pdfpages}
\usepackage{epstopdf}
\usepackage{booktabs}
\usepackage{authblk}

\usepackage{amsthm,bbm,amssymb}
\usepackage{amsmath}
\usepackage{natbib}
\usepackage[colorlinks,citecolor=blue,urlcolor=blue,filecolor=blue,backref=page]{hyperref}
\usepackage{graphicx}

\usepackage[utf8]{inputenc}
\usepackage{amsmath,amssymb,natbib,xcolor,multicol,url,mhequ}
\usepackage{graphicx,bbm,xr}
\usepackage{booktabs,epstopdf,color}
\usepackage[space]{grffile}
\usepackage{lineno}
\usepackage{multirow}

\newcommand{\be}{\begin{equs}}
\newcommand{\ee}{\end{equs}}

\numberwithin{equation}{section}
\theoremstyle{plain}
\newtheorem{theorem}{Theorem}

\usepackage{lipsum}
\usepackage{amsfonts,amsmath}
\usepackage{tikz-cd}
\usepackage{epstopdf}
\usepackage{algorithm} 
\usepackage{algorithmic}  
\usepackage[algo2e,ruled,vlined]{algorithm2e}

\newcommand{\x}{\textbf{x}}
\newcommand{\y}{\textbf{y}}

\title{Dynamics of coordinate ascent variational inference: A case study in 2D Ising models}

\author[1]{Sean Plummer\thanks{splumme@tamu.edu}}
\author[1]{Debdeep Pati\thanks{debdeep@stat.tamu.edu}}
\author[1]{Anirban Bhattacharya\thanks{anirbanb@stat.tamu.edu}}

\affil[1]{Department of Statistics,  Texas A\&M University, College Station, Texas, 77843, USA}

\begin{document}
\maketitle

\begin{abstract}
Variational algorithms have gained prominence over the past two decades as a scalable computational environment for Bayesian inference. In this article, we explore tools from the dynamical systems literature to study convergence of coordinate ascent algorithms for mean field variational inference.  
Focusing on the Ising model defined on two nodes, we fully characterize the dynamics of the sequential coordinate ascent algorithm and its parallel version. We observe that in the regime where the objective function is convex, both the algorithms are  stable and exhibit convergence to the unique fixed point.  Our analyses reveal interesting {\em discordances} between these two versions of the algorithm in the region when the objective function is non-convex.  In fact, the parallel version exhibits a periodic oscillatory behavior which is absent in the sequential version. Drawing intuition from the Markov chain Monte Carlo literature, we {\em empirically} show that a parameter expansion of the Ising model, popularly called as the Edward--Sokal coupling,  leads to an enlargement of the regime of convergence to the global optima. 

\end{abstract}


\section{Introduction}
Variational Bayes (VB) is now a standard tool to approximate computationally intractable posterior densities. Traditionally this computational intractability has been circumvented using sampling techniques such as Markov chain Monte Carlo (MCMC). MCMC techniques are prone to be computationally expensive for high dimensional and complex hierarchical Bayesian models, which are prolific in modern applications. VB methods, on the other hand, typically provide answers orders of magnitude faster as they are based on optimization. Introduction to VB can be found in chapter 10 of \cite{bishop2006pattern} and chapter 33 of \cite{mackay2003information}. Excellent recent surveys can be found in \cite{Blei_2017,zhang2018advances}. 

The objective of VB is to find the best approximation to the posterior distribution from a more tractable class of distributions on the latent variables that is well-suited to the problem at hand. The best approximation is found by minimizing a divergence between the posterior distribution of interest and a class of distributions that are computationally tractable. The most popular choices for the discrepancy and the approximating class are the Kullback--Leibler (KL) divergence and the class of product distributions, respectively. This combination is popularly known as mean field variational inference, originating from mean field theory in physics \cite{parisi1988statistical}. Mean-field inference has percolated through a wide variety of disciplines, ranging from statistical mechanics, electrical engineering, information theory, neuroscience, cognitive sciences \cite{opper2001advanced} and more recently in deep neural networks \cite{gabrie2020mean}. While computing the KL divergence is intractable for a large class of distributions, reframing the minimization problem for maximizing the evidence lower bound (ELBO) leads to efficient algorithms. In particular, for conditionally conjugate-exponential family models, the optimal distribution for mean field variational inference can be computed by iteration of closed form updates. These updates form a coordinate ascent algorithm known as coordinate ascent variational inference (CAVI) \cite{bishop2006pattern}. 

Research into the theoretical properties of variational Bayes has exploded in the last few years. Recent theoretical work focuses on statistical risk bounds for variational estimate obtained from VB \cite{Alquier2016, pati2018statistical, yang2020alpha, cherief2018consistency}, asymptotic normality of VB posteriors \cite{wang2019frequentist} and extension to model misspecification  \cite{Alquier2016, wang2019variational}.
While much of the recent theoretical work focuses on statistical optimality guarantees, there has been less work studying the convergence of the CAVI algorithms employed in practice. Convergence of CAVI to the global optima is only known in special cases that depend heavily on model structure for normal mixture models \cite{wang2005inadequacy, wang2006convergence}, stochastic block models \cite{zhang2017theoretical,Mukherjee2018}, and topic models \cite{ghorbani2018instability}, as well as under special restrictions of the parameter regime for Ising models \cite{jain2018meanfield, koehler2019fast}. Convergence properties of the CAVI algorithm is still largely an open problem. 

The goal of this work is to suggest a general systematic framework for studying convergence properties of CAVI algorithms. By viewing CAVI as a discrete time dynamical system, we can leverage dynamical systems theory to analyze the convergence behavior of the algorithm and bifurcation theory to study the types of changes that solutions can undergo as the various parameters are varied. For sake of concreteness, we focus on the 2d Ising model. Our contribution to the literature is as follows: We provide a complete classification of the dynamical properties of the the traditional sequential update CAVI algorithm as well as a parallelized version of the algorithm using dynamical systems and bifurcation theory on the Ising models. Our findings show that the sequential CAVI algorithm and the parallelized version have different convergence properties. Additionally, we numerically investigate the convergence of the CAVI algorithm on the Edward--Sokal coupling, a generalization of the Ising model. Our findings suggest that couplings/parameter expansion may provide a powerful way of controlling the convergence behavior of the CAVI algorithm, beyond the immediate example considered here.

\section{Mean-field variational inference and the coordinate ascent algorithm}\label{ssec:rev}
In this section, we briefly introduce mean-field variational inference for 
a  target distribution in the form of a Boltzmann distribution with potential function $\Psi$, 
\begin{eqnarray*}
p(\x) = \frac{\exp\{ \Psi(\x) \}}{\mathcal{Z}}, \quad \x \in \mathcal{X},
\end{eqnarray*}
where $\mathcal{Z}$ denotes the intractable normalizing constant. The above representation encapsulates both posterior distributions that arise in Bayesian inference, where $\Psi$ is the log-posterior upto constants, as well as probabilistic graphical models such as the Ising and Potts models. For instance, $\Psi(x) = \beta \sum_{u\sim v}J_{uv}x_u x_v +\beta \sum_{u} h_u x_u$ for the Ising model; see the next section for more details. Much of the complications in inference arise from the intractability of the normalizing constant $\mathcal{Z}$, which is commonly referred to as the free energy in probabilistic graphical models, and the marginal likelihood or evidence in Bayesian statistics. Variational inference aims to mitigate this problem by using optimization to find a best approximation $q^*$ to the target density $p$ from a class $\mathcal{F}$ of variational distributions over the parameter vector $\x$,
\begin{eqnarray}\label{eq:VI_optimization}
q^* = \arg\min_{q\in \mathcal{F}} D( q \,||\, p)
\end{eqnarray}
where $D( q \,||\, p)$ denotes the Kullback-Leibler (KL) divergence between $q$ and $p$. The complexity of this optimization problem is largely determined by the choice of variational family $\mathcal{F}$. The objective function of the above optimization problem is intractable because it also involve the evidence $\mathcal{Z}$. We can work around this issue by rewriting the KL divergence as 
\begin{eqnarray}\label{eq:equivalent_KL}
D( q \,||\, p) = \mathbb{E}_q[\log q] -\mathbb{E}_q[\log p] + \log \mathcal{Z}
\end{eqnarray}
where $\mathbb{E}_q$ denotes the expectation with respect to $q(\x)$.  Rearranging terms,
\begin{eqnarray}
\log \mathcal{Z} &= D( q \,||\, p)  +\mathbb{E}_q[\log p] -\mathbb{E}_q[\log q] \\
&\geq \mathbb{E}_q[\log p] -\mathbb{E}_q[\log q] :=\text{ELBO}(q).
\end{eqnarray}
The acronym ELBO stands for evidence lower bound and the nomenclature is now apparent from the above inequality. Notice from equation \eqref{eq:equivalent_KL} that maximizing the ELBO is equivalent to minimizing the KL divergence. By maximizing the ELBO we can solve the original variational problem while by-passing the computational intractability of the evidence.  

As mentioned above, the choice of variational family controls both the complexity and accuracy of approximation. Using a more flexible family achieves a tighter lower bound but at the cost of having to solve a more complex optimization problem. A popular choice of family that balances both flexibility and computability is the mean-field family. 
Mean-field variational inference refers to the situation when $q$ is restricted to the product family of densities over the parameters,  
\small \begin{eqnarray}\label{eq:MF}
    \mathcal{F}_{\text{MF}} :\, = \big\{ q(\x) = q_1(x_1) \otimes \cdots \otimes q_n(x_n) \,\, \text{for probability measures} \,\,  q_j, j=1, \ldots, n\big\},
\end{eqnarray}
\normalsize

\begin{algorithm}[H]
    \SetAlgoLined
    \KwIn{Model $p(\x)=\exp(\Psi(\x)-\log\mathcal{Z})$}
    \KwOut{A variational density $q(\x)=\prod_{j=1}^nq_j(x_j)$}
    \textbf{Initialize:} variational densities $q_j(x_j)$\;
    
    \While{$\text{ELBO}(q)$ not converged}{
        \For{$j\,\in\{1,\ldots,n\}$}{
        $q_j(x_j) \propto \exp\left\{\mathbb{E}_{-j}\left[ \log p(\x_j,\x_{-j})\right]\right\} $
        }
        Compute $\text{ELBO}(q)=\mathbb{E}_q[\log p(\x)] -\mathbb{E}_q[\log q(\x)] $
    }
    \Return{$q(\x)$}
\caption{Coordinate ascent variational inference ({ \sc CAVI})}
\label{algo:CAVI}
\end{algorithm}

 The coordinate ascent variational inference (CAVI) algorithm (refer to Algorithm \ref{algo:CAVI}) is a learning algorithm that optimizes the ELBO over the mean-field family $\mathcal{F}_{\text{MF}}$.  At each time step $t\geq 1$, the CAVI algorithm iteratively updates the current mean field marginal distribution $q_j^{(t)}(x_j)$ by maximizing the ELBO over that marginal while keeping the other marginals $\{q_\ell^{(t)}(x_\ell)\}_{\ell \neq j}$ fixed at their current values. Formally, we update the current distribution $q^{(t)}(\x)$ to $q^{(t+1)}(\x)$ by the updates,
 \begin{eqnarray*}
 q_1^{(t+1)}(x_1) &=& \arg\max_{q_1} \text{ELBO}(q_1\otimes q_2^{(t)}\otimes \cdots \otimes q_n^{(t)})\\
 q_2^{(t+1)}(x_2) &=& \arg\max_{q_2} \text{ELBO}(q_1^{(t+1)}\otimes q_2\otimes q_3^{(t)}\otimes \cdots \otimes q_n^{(t)})\\
 & \vdots& \\
 q_n^{(t+1)}(x_n) &=& \arg\max_{q_n} \text{ELBO}(q_1^{(t+1)}\otimes \cdots \otimes q_{n-1}^{(t+1)}\otimes q_n). 
 \end{eqnarray*}
 
The objective function $\text{ELBO}(q_1\otimes \cdots \otimes q_n)$ is concave in each of the arguments individually (although it is rarely jointly concave), so these individual maximization problems have unique solutions. The optimal update for the $j$th mean field variational component of the model has the closed form,  
\begin{eqnarray*}
q_j^*(\x_j) &\propto& \exp\left\{\mathbb{E}_{-j}\left[ \log p(\x_j,\x_{-j})\right]\right\}
\end{eqnarray*}
where the expectation $\mathbb{E}_{-j}$ are taken with respect to the distribution $\prod_{i\neq j}q_i(\x_i)$. Furthermore the update steps of the algorithm is monotone as each step of the CAVI iterates increases the objective function 
\begin{eqnarray*}
\text{ELBO}(q_1^{(t+1)}\otimes q_2^{(t+1)}\otimes \cdots \otimes q_n^{(t+1)}) &\geq \text{ELBO}(q_1^{(t+1)}\otimes q_2^{(t+1)}\otimes \cdots \otimes q_n^{(t)}) \geq \cdots \geq \\ & \text{ELBO}(q_1^{(t)}\otimes q_2^{(t)}\otimes \cdots \otimes q_n^{(t)}).
\end{eqnarray*}
  
For parametric models, the sequential updates of the variational marginal distributions in the CAVI algorithm is done by a sequential system updates of the variational parameters of these distributions. The CAVI algorithm updates for parametric models induce a discrete time dynamical system of the parameters. Clearly, convergence of the CAVI algorithm can be framed in terms of this induced discrete time dynamical system. As discussed before,  the ELBO is generally a non-convex function and hence the CAVI algorithm will only guaranteed to converge to a local optimum of the system. It is also not clear how many local optima (or fixed points) the system has or whether the algorithm always settle on a single fixed point, diverge away from the fixed point or  cycle between multiple fixed points.  These questions translate to questions about the existence and stability of fixed points of the induced dynamical system. We are also interested in how the behavior of the CAVI algorithm could possibly change as we vary the parameters of the model. This translates to questions about the possible bifurcations of the induced dynamical system. In Section \ref{sec:cavi_update}, we formally introduce the Ising model and its mean-field variational inference. 

\section{CAVI in Ising model}
\label{sec:cavi_update}
We first briefly review the definition of an Ising model. The Ising model was first introduced as a model for magnetization in statistical physics but has found many applications in other fields; see \cite{friedli2017statistical} and references therein. The Ising model is a probability distribution on the hypercube $\{\pm1\}^n$ given by
\begin{eqnarray} \label{eq:ising2}
p(\x)&\propto&\exp\left[\beta\sum_{u\sim v} J_{uv}x_u x_v + \beta \sum_u h_u x_u\right],
\end{eqnarray}
where the interaction matrix $J$ is a symmetric real $n \times n$ matrix with zeros on the diagonal, $h$ is a real $n$-vector that represents the external magnetic field, and $\beta$ is the \textit{inverse temperature} parameter. The model is said to be \textit{ferromagnetic} if $J_{uv}\geq 0$ for all $u,v$ and \textit{anti-ferromagnetic} if $J_{uv}<0$ for all $u,v$. The normalizing constant or the \textit{partition function} of the Ising model is 
\begin{eqnarray*}
\mathcal{Z}&=&\sum_{\x\in\{\pm1\}^n}\exp\left[\beta\sum_{u\sim v} J_{uv}x_u x_v + \beta \sum_u h_u x_u \right]. 
\end{eqnarray*}
Refer to Chapter 31 of \cite{mackay2003information} for an excellent review of Ising models.

\subsection{Mean field variational inference in Ising model}\label{ssec:mf}
Here we provide a derivation of the CAVI update function for the Ising model, focusing on the two nodes ($n=2$) case for simplicity and analytic tractability. 
\begin{figure}[htbp!]
    \centering
    \includegraphics[width=12cm]{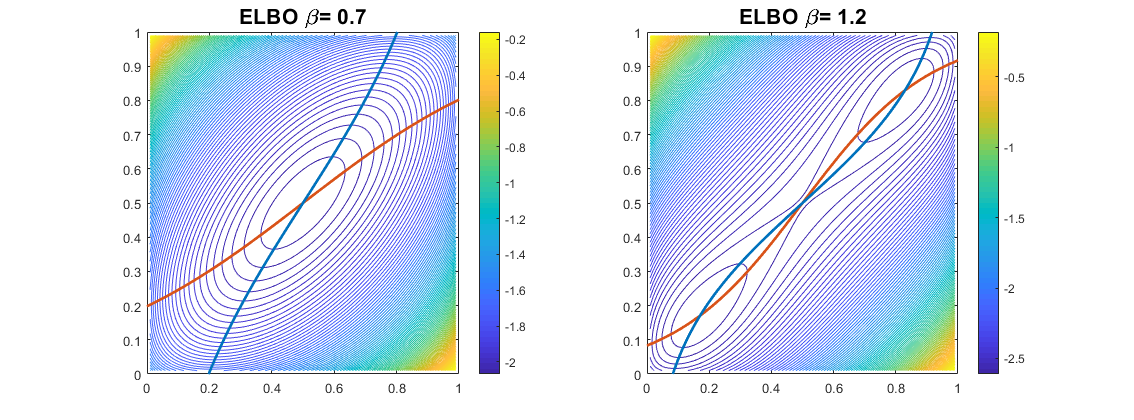}
    \caption{Left: A contour plot of the ELBO function for the Ising model with $\beta=0.7$ along with the marginal gradient functions. The y-marginal is orange and x-marginal is blue. For $0\leq \beta\leq 1$ there is one minimum of the ELBO. The circle patterns seen at the boundary are artifacts of scaling of the image. Right: A contour plot of the ELBO function for the Ising model with $\beta=1.2$ along with the marginal gradient functions. The y-marginal is orange and x-marginal is blue. For $1< \beta$ there are two minima of the ELBO. }
    \label{fig:CAVI_contour05}
\end{figure}
\begin{figure}[htbp!]
    \centering
    \includegraphics[width=12cm]{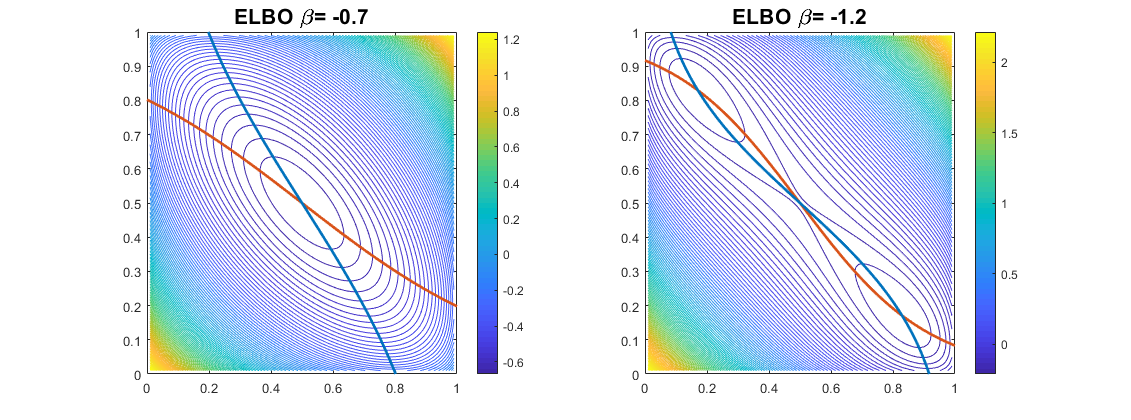}
    \caption{Left: A contour plot of the ELBO function for the Ising model with $\beta=-0.7$ along with the marginal gradient functions. The y-marginal is orange and x-marginal is blue. For $-1\leq \beta\leq 0$ there is one minimum of the ELBO. The circle patterns seen at the boundary are artifacts of scaling of the image. Right: A contour plot of the ELBO function for the Ising model with $\beta=-1.2$ along with the marginal gradient functions. The y-marginal is orange and x-marginal is blue. For $\beta< -1$ there are two minima of the ELBO.}
    \label{fig:CAVI_contour12}
\end{figure}
 Notice $\log p(\x):= \beta \mathcal{H}(\x)=\beta \sum_{u\sim v}J_{uv}x_u x_v +\beta\sum_{u} h_u x_u$.  
In this case we have the Ising model on two spins with $\x=(x_1,x_2)$ and influence matrix $J$ with off diagonal term $J_{12}$ and external magnetic field $h=(h_1,h_2)=(0,0)$. 
From the general framework in Section 2, the CAVI updates are given by,
$$
q_j^*(x_j)\propto \exp\left\{\mathbb{E}_{-j}\left[ \log p(x_j,x_{-j})\right]\right\}. 
$$
Equivalently, the same updates are obtained by setting the gradient of the ELBO as a function of $(x_1, x_2)$ to the $(0, 0)'$ vector. 
An illustration of the ELBO and the gradient functions for various values of $\beta$ is in Figures \ref{fig:CAVI_contour05} and \ref{fig:CAVI_contour12} respectively. 

Since $q_1^*$ and $q_2^*$ are two point distributions, it is sufficient to keep track of the mass assigned to 1. Simplifying, 
\begin{eqnarray*}
q_1^*(x_1)&\propto& \exp\left\{\mathbb{E}_{2}\left[ \log p(x_1,x_{2})\right]\right\}\\
&=&\exp\left\{\beta\mathcal{H}(x_1,x_2=1) q_2(x_2=1)+\beta\mathcal{H}(x_1,x_2=-1) q_2(x_2=-1)\right\}\\
&=&\exp\left\{(\beta J_{12}x_1+ \beta h_1 x_1 + \beta h_2) \xi +(-\beta J_{12}x_1 + \beta h_1x_1 - \beta h_2)(1-\xi) \right\}\\
&=&\exp\left\{(2\xi-1)(\beta J_{12}x_1+ \beta h_2) + \beta h_1 x_1 \right\},
\end{eqnarray*}
where $\xi=q_2(x_2=1)$.
Therefore 
\begin{eqnarray*}
q_1^*(x_1=1)&=&\frac{\exp\left\{(2\xi-1)(\beta J_{12} + \beta h_2) + \beta h_1\right\}}{\exp\left\{(2\xi-1)(\beta J_{12} + \beta h_2)+\beta h_1 \right\}+\exp\left\{(2\xi-1)(-\beta J_{12}  + \beta h_2) - \beta h_1 \right\}}\\
&=&\frac{1}{1+\exp\left\{-2\beta J_{12}(2\xi-1) - 2\beta h_1 \right\}}.
\end{eqnarray*}
Similarly denoting $\zeta=q_1(x_1=1)$,
\begin{eqnarray*}
q_2^*(x_2=1)&=&\frac{\exp\left\{(2\zeta-1)(\beta J_{12}+\beta h_1 ) + \beta h_2\right\}}{\exp\left\{(2\zeta-1)(\beta J_{12}+\beta h_1 )+ \beta h_2  \right\}+\exp\left\{(2\zeta-1)(-\beta J_{12}+\beta h_1 ) - \beta h_2 \right\}}\\
&=&\frac{1}{1+\exp\left\{-2\beta J_{12}(2\zeta-1)- 2\beta h_2\right\}}. 
\end{eqnarray*}

Let $\zeta_k$ (resp. $\xi_k$) denote the $k$th iterate of $q_1(x_1=1)$ (resp. $q_2(x_2=1)$) from the CAVI algorithm. To succinctly represent these updates, define the logistic sigmoid function 
\begin{eqnarray}\label{eq:logistic_sigmoid}
\sigma(u,\beta)=\frac{1}{1+e^{-\beta u}},\quad u\in[0,1],\quad \beta\in \mathbb{R}. \label{sigmoid}
\end{eqnarray}	
With this notation, we have, for any $k \in \mathbb{Z}_+$, 
\begin{eqnarray}
\begin{aligned}
\label{sequential.cavi}
\zeta_{k+1} &=\sigma(J_{12}(2\xi_{k}-1) + h_1, 2\beta)\\
\xi_{k+1} &=\sigma(J_{12}(2\zeta_{k+1}-1) + h_2, 2\beta). 
\end{aligned}
\end{eqnarray}
Without loss of generality we henceforth set $J_{12}= 1$. Under this choice the model is in the ferromagnetic regime for $\beta>0$ and the anti-ferromagnetic regime for $\beta<0$.

\section{Why the Ising model: a summary of our contributions}
 There are exactly two cases of the Ising model that have a full analytic solution for the free energy. They are i) the one dimensional line graph solved by Ernst Ising in his thesis \cite{ising1925beitrag} and  ii) the two dimensional case on the anisotropic square lattice when the magnetic field $h = 0$ by \cite{Onager1944}. Comparison with the mean field solution for the same models highlights the poor approximation quality of the mean field solution in low dimensions. To the best knowledge of the authors, there are no results in the literature detailing the properties of the mean field solution to the anti-ferromagnetic Ising model. Readers not familiar with the physics may wonder why this is the case. To explain this, there are two cases in the anti-ferromagnetic regime: one of the two regions is equivalent to the ferromagnetic case and in the other the mean field approximation is not a good approximation of the system. The first case occurs in a bipartite graph where a transformation of variables makes the antiferromagnetic regime equivalent to the ferromagnetic one \cite{toda2012statistical}. The other case can be seen on the triangle graph. By fixing the spin of one vertex as $1$ and the other as $-1$, the third vertex becomes geometrically frustrated and neither choice of spin lowers the energy level of the system and the two configurations are equivalent \cite{moessner2006geometrical}. In this case the mean field approximation gives completely incorrect answer and does not merit further investigation from a qualitative point of view. The physics literature is primarily concerned with using the mean field solutions to the Ising model to estimate important physical constants of the systems. These constants are only meaningful when the mean field solution provides a good approximation to the behavior of the system in large dimensions. It is known however that under certain conditions the mean field approximation does indeed converge to the true free energy of the system as the dimension increases \cite{jain2018meanfield, Basak2017}. 

Our work is focused on providing a rigorous methodology to analyze dynamics of the CAVI algorithm that can be applied to any model structure. All of the interesting behaviors exhibited by the CAVI algorithm fit into the classical mathematical framework of discrete dynamical systems and bifurcation theory. Specifically we use the Ising model as a simple and yet rich example to illustrate the potential of dynamical systems theory to analyze CAVI updates for mean field variational inference. The bifurcation of the ferromagnetic Ising model at the boundary of the Dobrushin regime is known \cite{mackay2003information, friedli2017statistical}, however a rigorous proof in terms of dynamical systems theory is missing in the literature.

There are several features that make the CAVI algorithm on the Ising model a nontrivial example worth investigating. The optimization problem arising from mean field variational inference on the Ising model is, in general, non-convex \cite{jain2018meanfield}. However it is straightforward to obtain sufficient conditions to guarantee the existence of a global optima. One such condition is that the inverse temperature $\beta$ is inside the Dobrushin regime, $\vert \beta \vert < 1$ \cite{jain2018meanfield}. Inside the Dobrushin regime, the CAVI update equations form a contraction mapping guaranteeing a unique global optima \cite{jain2018meanfield}. Outside of this regime the behavior of the CAVI algorithm is nontrivial. The CAVI solution to the Ising model with zero external magnetic field exhibits multiple local optima outside of the Dobrushin regime \cite{mackay2003information}. 

Our contributions to the literature are as follows. We utilize tools from dynamical systems theory to rigorously classify the full behavior of Ising model for the full parameter regime in dimension $n=2$ for both the sequential and parallel versions of CAVI algorithm. We show that the dynamical behavior of the sequential CAVI is not equivalent to the behavior of the parallel CAVI. Lastly we derive a variational approximation to the Edward-Sokal parameter expansion of the Potts and Random Cluster models and numerically study its convergence behavior under the CAVI algorithm. Our numerical results reveal that the parameter expansion leads to an enlargement of the regime of convergence to the global optima. In particular the Dobrushin regime is strictly contained in the expanded regime. This is compatible with the analogous results in Markov chain literature. See the introduction of \cite{blanca2018} for a well written summary of Markov chain mixing in the Ising model. 

\subsection{Statistical Significance of our results}
Although mean field variational inference has been routinely used in applications \cite{Blei_2017} for computational efficiency, it may not yield statistically optimal estimators \cite{ghorbani2018instability, wang2005inadequacy}. Mean field inference approximates the joint probability mass function in \eqref{eq:ising2} for $n=2$ by product of two distributions on $\{-1, 1\}$ in the sense of Kullback--Leibler divergence.  As discussed in \S \ref{ssec:mf}, minimizing this divergence is equivalent to maximizing an objective function, called the Evidence Lower Bound (ELBO). Our objective is to better understand the relation between the CAVI estimate and the global maximum of ELBO  in \eqref{eq:ising2} when $n=2$ and $h=0$.   Ideally, we want the global maximum of the ELBO to be a statistically reliable estimate. To understand this, let us denote $2 \times \mbox{Bernoulli}(p)-1$ by $\langle 1, -1; p \rangle$. As the marginal distributions of  \eqref{eq:ising2}  are both equal to $\langle 1, -1; 0.5 \rangle$, we want the ELBO to be maximized at this value. From an algorithmic perspective, we would like to ensure that the CAVI iterates converge to this global maximum. The synergy of these two phenomena leads to a successful variational inference method. We showed in this article that both these conditions can be violated in a certain regime of the parameter space in the context of Ising model on two nodes. Inside the Dobrushin regime ($-1\leq \beta \leq 1$), the global optima of the ELBO obtained from a mean field inference occurs at $(\langle 1, -1; 0.5 \rangle, \langle 1, -1; 0.5 \rangle)$ which is qualitatively the optimal solution. In this regime, the CAVI system converges to this global optimum irrespective of where the system is initialized.  Thus the mean field inference yields statistically optimal estimate and the algorithm is stable and convergent at this value. Unfortunately, this property deteriorates outside of the Dobrushin regime. Outside of the regime, the global maxima occur at two symmetric points which are different from $(\langle 1, -1; 0.5 \rangle, \langle 1, -1; 0.5 \rangle)$. These two symmetric points are equivalent under label switching. For example, when $\beta=1.2$ one of the optima is $(\langle 1, -1; 0.17071 \rangle, \langle 1, -1; 0.17071 \rangle)$ and the other is $(\langle 1, -1; 0.82928 \rangle, \langle 1, -1; 0.82928 \rangle)$. Notice this second optima is equivalent to the sign swapped version $(\langle -1, 1; 0.17071 \rangle, \langle -1, 1; 0.17071 \rangle)$. 

The original optima $(\langle 1, -1; 0.5 \rangle, \langle 1, -1; 0.5 \rangle)$  is actually a local minimum of the ELBO outside the Dobrushin regime. We illustrate in our theory that the CAVI system returns one of two global maxima of the objective function depending on the initialization of the algorithm. Although it is widely known that the statistical quality of the mean field inference is poor outside the regime, we show in addition that the algorithm itself exhibits erratic behavior and may not converge to the global maximizer of the ELBO for all initializations. Interestingly, outside the Dobrushin regime, the statistically optimal solution $(\langle 1, -1; 0.5 \rangle, \langle 1, -1; 0.5 \rangle)$ is a repelling fixed point of the CAVI system. This means that as the system is iterated, the current value of the system is pulled away from $(\langle 1, -1; 0.5 \rangle, \langle 1, -1; 0.5 \rangle)$ to the global maximum. 

A common technique to further improve computational time is the use of block updates in the CAVI algorithm, meaning groups of parameters are updated simultaneously.  We refer to this as the parallelized CAVI algorithm. This has been shown to work well in certain models \cite{Mukherjee2018}, but has not been investigated in a general setting. However, it turns out that block updating in the Ising model can lead to new problematic behaviors. Outside the Dobrushin regime, block updates can exhibit non-convergence in the form of cycling. As the system updates, it eventually switches back and forth between two points that yield the same value in the objective function. 

Parameter expansions (coupling) is another method of improving the convergence properties of algorithms. In the Markov chain theory for Ising models, it is well-known that mixing and convergence time are typically improved by using the Edward--Sokal coupling, a parameter expansion of the ferromagnetic Ising model \cite{guo2018}. Our preliminary investigation reveals that the convergence properties of the CAVI algorithm also exhibit a similar phenomenon.

\section{Main Results}\label{sec:main}
In this section, we analyze the behavior of the dynamical systems that one can form using the CAVI update equations and show that the behaviors of the systems differ.  Our results are heavily dependent on well-known techniques in dynamical systems. For readers unfamiliar with some of technical terminology below, we have included a primer on the basics of dynamical systems in Appendix \ref{sec:dynamical}.
 
Recall the system of sequential updates, which are the updates used in CAVI
\begin{eqnarray}
\zeta_{k+1}=\sigma(2\xi_k-1,2\beta), \quad \xi_{k+1}=\sigma(2\zeta_{k+1}-1,2\beta), \label{sys:seq.cavi}
\end{eqnarray}	
and the parallel updates 
\begin{eqnarray}
\zeta_{k+1}=\sigma(2\xi_k-1,2\beta), \quad \xi_{k+1}=\sigma(2\zeta_{k}-1,2\beta).\label{sys:par.cavi}
\end{eqnarray}	
We will show that these two systems are not topologically conjugate. We first state and prove some results on the dynamics of the sigmoid function \eqref{eq:logistic_sigmoid}. These results will as building blocks to study the dynamics of \eqref{sys:seq.cavi} and \eqref{sys:par.cavi}.   Phase change behavior of dynamical systems using the sigmoid and RELU activation functions are known in the literature in the context of generalization performance of deep neural networks \cite{oostwal2019hidden,ccakmak2020dynamical}.  In this section we present a complete proof of the bifurcation analysis of non-linear dynamical systems involving sigmoid activation function despite its connections with \cite{oostwal2019hidden,ccakmak2020dynamical}.  Our results in Section \ref{eq:sigdyn} provide a more complete picture of the behavior of the dynamics in all regimes and can be readily exploited to analyze the dynamics of \eqref{sys:seq.cavi} and \eqref{sys:par.cavi}. 
\subsection{Sigmoid Function Dynamics}\label{eq:sigdyn}
The proof for the dynamics of the CAVI algorithm on the Ising model relies heavily on the dynamics of the following sigmoid function and its second iterate
\begin{eqnarray}
\begin{aligned} \label{fn:sigmoid.2}
\sigma(2x-1,2\beta) \label{fn:sigmoid.1}, \quad 
\sigma(2\sigma(2x-1,2\beta)-1,2\beta).
\end{aligned}
\end{eqnarray}
Using numerical techniques, we solve for the number of fixed points of the system. The number of fixed points the function (\ref{fn:sigmoid.1}) depends on the magnitude of the parameter. In the ferromagnetic regime there is no periodic behavior, so there are no additional fixed points in (\ref{fn:sigmoid.2}) that are not fixed points in (\ref{fn:sigmoid.1}). For $\beta \leq 1$, there is a single fixed point at $x_*=1\slash 2$ and for $\beta>1$, there are 3 fixed points $c_0(\beta), 1\slash 2, c_1(\beta)$  in the interval $[0,1]$. These fixed points satisfy $0\leq c_0(\beta)< 1\slash 2<c_1(\beta)\leq 1$, $c_0(\beta)\to 0$ and $c_1(\beta)\to 1$ as $\beta\to \infty$. In the anti-ferromagnetic regime we see periodic behavior in the system; there are fixed points of (\ref{fn:sigmoid.2}) that are not fixed points of (\ref{fn:sigmoid.1}). For $\beta<1$, there is only one fixed point at $x_*=1\slash 2$ 
and a periodic cycle $C=\{c_0(\beta),c_1(\beta)\}$. Both $c_0(\beta),c_1(\beta)$ are fixed points of (\ref{fn:sigmoid.2}) and these points are the same fixed points from the ferromagnetic regime as (\ref{fn:sigmoid.2}) is an even function with respect to $\beta$. 

The following table denotes the values of the derivatives at the fixed point $1\slash2$ for $\beta=\pm 1$. 
\begin{table}[htbp]\label{table:derivatives}
{\footnotesize
  \caption{Derivatives of (\ref{fn:sigmoid.1}) and (\ref{fn:sigmoid.2}) at fixed point $x_*=1\slash 2$ for parameter value $\beta=\pm 1$. Here $\sigma^2$ is shorthand notation for (\ref{fn:sigmoid.2})}  \label{tab:derivative}
\begin{center}
  \begin{tabular}{|c|c|c|c|c|c|c|c|c|c|c|} \hline
    &$\sigma_x$ & $\sigma_{xx}$ & $\sigma_{xxx}$ &$ \sigma_{\beta}$ & $\sigma_{\beta x}$& $\sigma^2_x$ & $\sigma^2_{xx}$ & $\sigma^2_{xxx}$ & $\sigma^2_{\beta}$ & $\sigma^2_{\beta x}$ \\ \hline
    $\beta=1$ & 1 & 0 & -8 & 0 & $1\slash 2$ & 1  & 0 & -16 &  0 &  1\\ \hline
    $\beta=-1$ &  -1 & 0 & 8 & 0 & $1\slash 2$& 1 &  0 & -16 &  0 &  -1\\ \hline
  \end{tabular}
\end{center}
}
\end{table}

We now have enough information to provide a complete classification of the dynamics of the sigmoid function.
\begin{theorem}[Dynamics of sigmoid function]\label{thm:sigmoid.dynamics}
Consider the discrete dynamical system generated by (\ref{fn:sigmoid.1})
\begin{eqnarray*}
x\mapsto \sigma(2x-1,2\beta)=\frac{1}{1+e^{-2\beta(2x-1)}}.
\end{eqnarray*}
The full dynamics of the system (\ref{fn:sigmoid.1}) are as follows 
	\begin{enumerate}
		\item For $-1\leq \beta\leq 1$, the system has a single hyperbolic fixed point $x_*=1\slash 2$ which is a global attractor and there are no $p$-periodic points for $p\geq 2$. 
		\item For $\beta>1$, the system has one repelling hyperbolic fixed point $x_*=1 \slash 2$ and two hyperbolic stable fixed points $c_0$, $c_1$, with $0<c_0< 1\slash 2< c_1<1$, and stable sets $W^s(c_0)=[0,1\slash 2)$, $W^s(c_1)=(1\slash 2,1]$. There are no $p$-periodic points for $p\geq 2$. 
		\item For $\beta<-1$, the system has one unstable hyperbolic fixed point $x_*=1\slash2$, and a stable 2-cycle $\mathcal{C}=\{c_0,c_1\}$ with stable set $W^s(\mathcal{C})=[0,1\slash2)\cup (1\slash 2, 1]$, with $0<c_0<1\slash 2< c_1<1$. There are no $p$-periodic points for $p>2$. 
		\item For $\vert \beta \vert = 1$, the system has one non-hyperbolic fixed point at $x_*=1\slash 2$ which is asymptotically stable and attracting. 
	\end{enumerate}
The system undergoes a PD bifurcation at $\beta=-1$ and a pitchfork bifurcation at $\beta=1$. 
\end{theorem}

\begin{proof}
We will break the proof up into three parts. The first part of the proof is a linear stability analysis of the system, the second part is a stability analysis of the periodic points in the system, and the third part is an analysis of the bifurcations of the system. We begin with a linear stability analysis of the system at each fixed point. For $\beta \leq 1$ the system has one fixed point $x_*=1\slash 2$ and for $\beta>1$ the system has three fixed points $c_0, 1\slash 2, c_1$. The derivative of $\sigma(2x-1,2\beta)$ is $\sigma_x(2x-1,2\beta)=-4\beta\sigma(2x-1,2\beta)(1-\sigma(2x-1,2\beta))$. 

\textbf{Fixed point $x_*=1\slash 2$} : The Jacobian of the system at the fixed point $x_*=1\slash 2$ is $\sigma_x(2x_*-1,2\beta)= \beta$. For $\beta\neq 1$, the fixed point $x_*=1\slash 2$ is hyperbolic and for $\beta=\pm 1$ the fixed point is non-hyperbolic. We classify the stability of the hyperbolic fixed point $x_*=1\slash 2$ using theorem \ref{thm:hyperbolic.stability}. For $\vert \beta \vert <1$ the fixed point $x_*=1\slash 2$ is globally attracting as $\vert \sigma_x(2x_*-1,2\beta)\vert <1$ and for $\vert \beta\vert >1$ the fixed point $x_*=1\slash 2$ is globally repelling as $\vert \sigma_x(2x_*-1, 2\beta_*)\vert > 1$. For $\beta=\pm 1$ we invoke theorem \ref{thm:nonhyperbolic} to check for stability of the fixed point. At $\beta=-1$ we have $\sigma_{x}(2x_*-1,2\beta)=-1$ and we need to check the Schwarzian derivative. The fixed point $x_*=1\slash 2$ is asymptotically stable for $\beta=-1$ by theorem \ref{thm:nonhyperbolic}, as $\mathcal{S}\sigma(2\sigma(2x-1,2\beta)-1,2\beta)\mid_{x=x_*}=-8$. For $\beta=1$ we have $\sigma_x(2x_*-1,2\beta)=1$ and we need to check the second and third derivatives at the fixed point. The fixed point $x_*=1\slash 2$ is asymptotically stable when $\beta=1$  by theorem \ref{thm:nonhyperbolic} as $\sigma_{xx}(2x_*-1,2\beta)=0$ and $\sigma_{xxx}(2x_*-1,2\beta)=-8$. 

\textbf{Fixed points $c_0,c_1$} : These fixed points have the same behavior so we have grouped them together in the analysis. When $\beta>1$ there are two additional fixed points $c_0,c_1$ of the system, both are attracting fixed points by theorem \ref{thm:hyperbolic.stability} as $\vert \sigma_x(2c_i-1, 2\beta)\vert < 1 $ for each $i=0,1$ and all $\beta>1$. The stable sets are $W^s(c_0)=[0,1\slash 2)$ and $W^s(c_1)=(1\slash 2, 1]$. 

\textbf{Periodic points}: For $\beta<-1$ we see the two cycle $\mathcal{C}=\{c_0,c_1\}$. Notice $\sigma(2c_0-1,2\beta)=c_1$ and $\sigma(2c_1-1,2\beta)=c_0$. This two cycle is stable since $c_0$ and $c_1$ are both stable fixed points of (\ref{fn:sigmoid.2}). The stable set is $W^s(\mathcal{C})=[0,1\slash2)\cup (1\slash 2, 1]$, \quad $0<c_0<1\slash 2< c_1<1$.

At $(x_*,\beta_*)=(1\slash 2, 1 )$ the system under goes a pitchfork bifurcation as it satisfies the conditions in theorem \ref{thm:pitchfork}
\begin{eqnarray*}
\sigma(2x_*-1,2\beta_*)=1\slash 2 \quad \sigma_x(2x_*-1,2\beta_*)=1\quad \sigma_{xx}(2x_*-1,2\beta_*)=0,\\
\sigma_\beta(2x_*-1,2\beta_*)=0\quad \sigma_{x\beta}(2x_*-1,2\beta_*)\neq0\quad \sigma_{xxx}(2x_*-1,2\beta_*)\neq0.
\end{eqnarray*}
Similarly at $(x_*,\beta_*)=(1\slash 2, -1)$ the system under goes a period doubling bifurcation as it satisfies the conditions in theorem \ref{thm:pdb}
\begin{eqnarray*}
\sigma(2x_*-1,2\beta_*)=1\slash 2\quad \sigma_x(2x_*-1,2\beta_*)=-1\quad \sigma_{xx}(2x_*-1,2\beta_*)=0,\\
\sigma_\beta(2x_*-1,2\beta_*)=0\quad \sigma_{x\beta}(2x_*-1,2\beta_*)\neq0\quad \sigma_{xxx}(2x_*-1,2\beta_*)\neq0. 
\end{eqnarray*}
\end{proof}

We can fully classify the dynamics of (\ref{fn:sigmoid.2}) using the above theorem. We omit the proof as it is similar to the proof of theorem \ref{thm:sigmoid.dynamics}.
\begin{theorem}
      \label{thm:sigmoid2.dynamics}
The full dynamics of the system (\ref{fn:sigmoid.2}) are as follows 
	\begin{enumerate}
		\item For $-1\leq \beta\leq 1$, the system has a single hyperbolic fixed point $x_*=1\slash 2$ which is a global attractor and there are no $p$-periodic points for $p\geq 2$. 
		\item For $\vert \beta \vert >1$, the system has one repelling hyperbolic fixed point $x_*=1 \slash 2$ and two hyperbolic stable fixed points $c_0$, $c_1$, with $0<c_0< 1\slash 2< c_1<1$, and stable sets $W^s(c_0)=[0,1\slash 2)$, $W^s(c_1)=(1\slash 2,1]$. 
		\item For $\vert \beta \vert = 1$, the system has one non-hyperbolic fixed point at $x_*=1\slash 2$ which is asymptotically stable and attracting. 
	\end{enumerate}
The system undergoes a pitchfork bifurcation at $\beta=\pm 1$. There are no $p$-periodic points for $p\geq 2$. 
\end{theorem}

\subsection{Sequential Dynamics}
To fully understand the dynamics of the equations defining the updates to $q^*_1$ and $q^*_2$ it suffices to track the evolution of the points $q^*_1(1)=\zeta$ and $q^*_2(1)=\xi$. The CAVI algorithm updates terms sequentially, using the new values of the variables to calculate the others. We initialize the CAVI algorithm at points $\zeta_0,\xi_0$. The CAVI algorithm is a dynamical system formed by sequential iterations of $\sigma(2x-1,2\beta)$ starting from $\zeta_0, \xi_0$. We can decouple the CAVI updates for $\xi_k$ and $\zeta_k$ by looking at the second iterations. This decoupling is visualized in the diagram 
\eqref{diagram:decoupling2} below.  The system formed the sequential updates is equivalent to the following decoupled system 
\begin{eqnarray*} 
\zeta_1 &&= \sigma(2\xi_0-1,2\beta)\\
\zeta_{k+1}&&=\sigma(2\sigma(2\zeta_k-1,2\beta)-1,2\beta),\quad k\geq 1 \\
\xi_{k+1}&&=\sigma(2\sigma(2\xi_k-1,2\beta)-1,2\beta), \quad k\geq 0. 
\end{eqnarray*}	
We propose to investigate dynamics of sequential systems by studying the dynamics of (\ref{fn:sigmoid.2}) in the Appendix \ref{sec:dynamical} acting on each variable individually. 
\begin{eqnarray}
\begin{tikzcd}
   \zeta_0 & \zeta_1 \arrow[d, "\sigma"] \arrow[r,"\sigma^2"]&  \zeta_2  \arrow[r,"\sigma^2"] \arrow[d, "\sigma"]& \zeta_3\cdots\\
  \xi_0 \arrow[r, "\sigma^2"] \arrow[ur, "\sigma"] & \xi_1 \arrow[r, "\sigma^2"] \arrow[ru, "\sigma"] & \xi_2 \arrow[r, "\sigma^2"] \arrow[ur, "\sigma"] & \xi_3\cdots
  \end{tikzcd}
  \label{diagram:decoupling2}
\end{eqnarray}
Illustrations of the evolution of the dynamics of the parallel updates for various initializations and values of $\beta$ are in Figures \ref{fig:1}-\ref{fig:4}. 
\begin{figure}[htbp!]
    \centering
    \includegraphics[width=10cm]{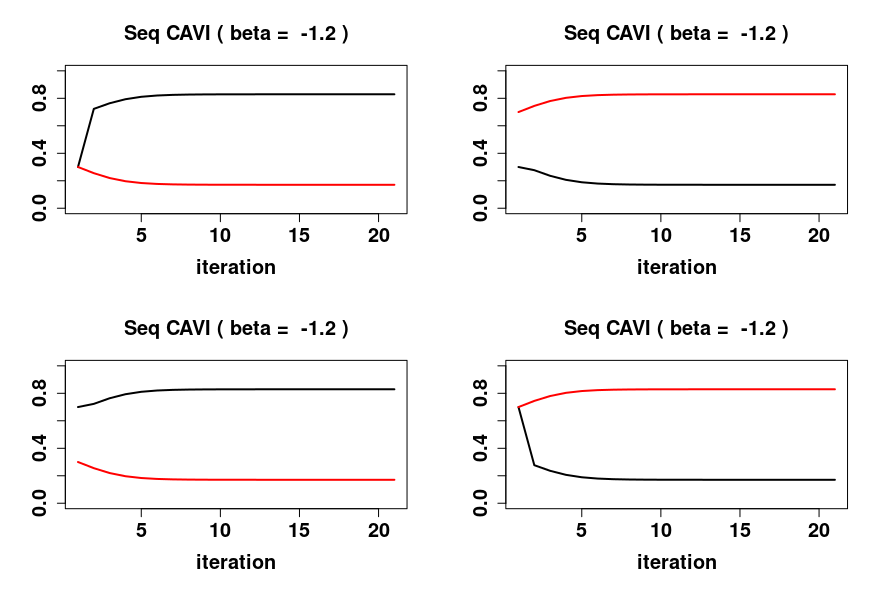}
    \caption{A plot of the first 20 iterations of the CAVI algorithm at various initializations for $\beta=-1.2$. In each of the plots the $\zeta$ updates are black and the $\xi$ updates are red. The upper left plot is an initialization of $\zeta_0=0.3$ and $\xi_0=0.3$, we see that $\zeta_k$ converge to the local fixed point $c_1(1.2)=0.82928$ and $\xi_k$ converge to the local fixed point $c_0(1.2)=0.17071$. The upper right is an initialization of $\zeta_0=0.3$ and $\xi_0=0.7$, we see that $\zeta_k$ converge to the local fixed point $c_0(1.2)=0.17071$ and $\xi_k$ converge to the local fixed point $c_1(1.2)=0.82928$. The lower left is is an initialization of $\zeta_0=0.7$ and $\xi_0=0.3$, we see that $\zeta_k$ converge to the local fixed point $c_1(1.2)=0.82928$ and $\xi_k$ converge to the local fixed point $c_0(1.2)=0.17071$. The upper left plot is an initialization of $\zeta_0=0.7$ and $\xi_0=0.7$, we see that $\zeta_k$ converge to the local fixed point $c_0(1.2)=0.17071$ and $\xi_k$ converge to the local fixed point $c_1(1.2)=0.82928$.}
    \label{fig:1}
\end{figure}

\begin{figure}[htbp!]
    \centering
    \includegraphics[width=10cm]{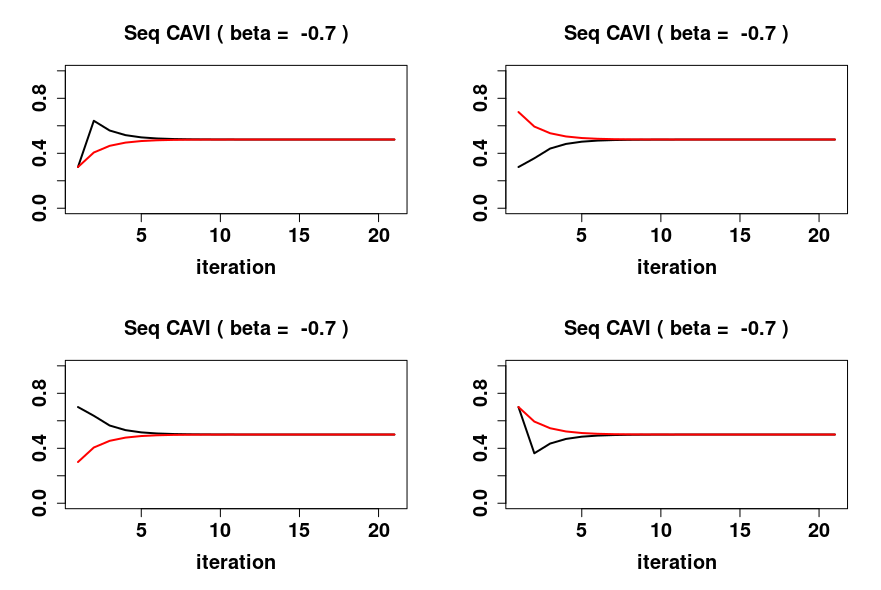}
    \caption{A plot of the first 20 iterations of the CAVI algorithm at various initializations for $\beta=-0.7$. In each of the plots the $\zeta$ updates are black and the $\xi$ updates are red. The upper left plot is an initialization of $\zeta_0=0.3$ and $\xi_0=0.3$, we see that both of these converge to the global fixed point $1\slash 2$. The upper right is an initialization of $\zeta_0=0.3$ and $\xi_0=0.7$, we see that this initialization converges to the global fixed point $1\slash 2$. The lower left is is an initialization of $\zeta_0=0.7$ and $\xi_0=0.3$, we see that this initialization converges to the global fixed point $1\slash 2$. The upper left plot is an initialization of $\zeta_0=0.7$ and $\xi_0=0.7$, we see that both of these converge to the global fixed point $1\slash 2$. }
    \label{fig:2}
\end{figure}

\begin{figure}[htbp!]
    \centering
    \includegraphics[width=10cm]{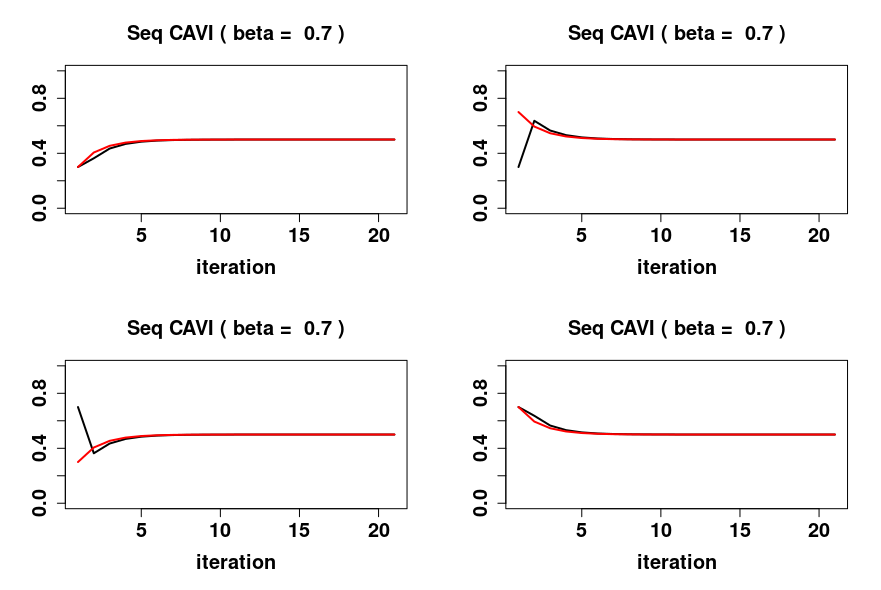}
    \caption{A plot of the first 20 iterations of the CAVI algorithm at various initializations for $\beta=0.7$. In each of the plots the $\zeta$ updates are black and the $\xi$ updates are red. The upper left plot is an initialization of $\zeta_0=0.3$ and $\xi_0=0.3$, we see that both of these converge to the global fixed point $1\slash 2$. The upper right is an initialization of $\zeta_0=0.3$ and $\xi_0=0.7$, we see that this initialization converges to the global fixed point $1\slash 2$. The lower left is is an initialization of $\zeta_0=0.7$ and $\xi_0=0.3$, we see that this initialization converges to the global fixed point $1\slash 2$. The upper left plot is an initialization of $\zeta_0=0.7$ and $\xi_0=0.7$, we see that both of these converge to the global fixed point $1\slash 2$.}
    \label{fig:3}
\end{figure}

\begin{figure}[htbp!]
    \centering
    \includegraphics[width=10cm]{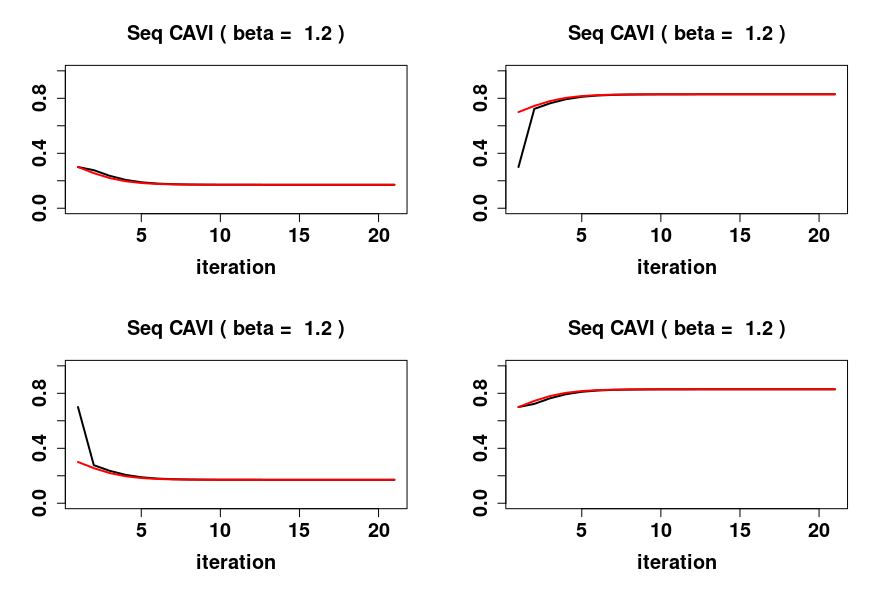}
    \caption{A plot of the first 20 iterations of the CAVI algorithm at various initializations for $\beta=1.2$. In each of the plots the $\zeta$ updates are black and the $\xi$ updates are red. The upper left plot is an initialization of $\zeta_0=0.3$ and $\xi_0=0.3$, we see that both of these converge to the local fixed point $c_0(1.2)=0.17071$. The upper right is an initialization of $\zeta_0=0.3$ and $\xi_0=0.7$, we see that this initialization converges to the local fixed point $c_1(1.2)=0.82928$. The lower left is is an initialization of $\zeta_0=0.7$ and $\xi_0=0.3$, we see that this initialization converges to the local fixed point $c_0(1.2)=0.17071$. The upper left plot is an initialization of $\zeta_0=0.7$ and $\xi_0=0.7$, we see that both of these converge to the local fixed point $c_1(1.2)=0.82928$}
    \label{fig:4}
\end{figure}

 \begin{theorem}[CAVI dynamics]\label{thm:sequential}
  	The dynamics of the CAVI system (\ref{sys:seq.cavi}) are given by
  	\begin{enumerate}
  	    \item For $ \beta < -1$ the system has two stable fixed points $(c_0,c_1)$ and $(c_1, c_0)$, with stable sets $W^s((c_0,c_1))= [0,1] \times [0,1\slash 2)$ and $W^s((c_1,c_0))= [0,1] \times (1\slash 2,1]$ respectively.
  	    \item For $\vert \beta \vert \leq 1$, the system has a global attracting fixed point $(1\slash2, 1\slash 2)$. 
  		\item For $\beta > 1$ the system has two stable fixed points $(c_0,c_0)$ and $(c_1, c_1)$, with domains of attraction $W^s((c_0,c_0))=[0,1]\times [0,1\slash 2)$ and $W^s((c_1,c_1))=[0,1]\times (1\slash 2,1]$ respectively.
  		\item For $\vert \beta \vert> 1$, the system has one unstable fixed point $(1\slash2, 1\slash 2)$.
  		\item For $\vert \beta \vert> 1$, the system has no $p$-periodic points for $p\geq 2$.
  	\end{enumerate}
    where $0\leq c_0<1\slash 2<c_1\leq 1$ are the fixed points of (\ref{fn:sigmoid.1}) in $[0,1]$. The system under goes a super-critical pitchfork bifurcation at $\beta= -1$ and again at $\beta= 1$ 
\end{theorem}

\begin{proof} 
We will analyze the CAVI system in the same way that we analyzed the sigmoid function. We begin by preforming a linear stability analysis of the fixed points in the system and analysing periodic points. We then show that the system satisfies the sufficient conditions of the pitchfork bifurcation.

For $\vert \beta \vert \leq 1$, the function (\ref{fn:sigmoid.2}) has one fixed point $x_*=1\slash 2$. For $\vert \beta\vert >1$ the function (\ref{fn:sigmoid.2}) has three fixed points $c_0$, $1\slash 2$, $c_1(\beta)$ where $0\leq c_0<1\slash 2<c_1\leq 1$. Calculating the Jacobian of the system at each of the fixed points shows that for $\vert \beta \vert<1$, the single fixed point $1\slash 2$ is a global attracting fixed point with $W^s(1\slash 2)=[0,1]$. For $\vert \beta \vert>1$, $x_*=1\slash 2$ is a repelling fixed point with $W^s(1\slash 2)=\{1\slash2\}$, $W^u(1\slash 2)=(c_0, 1\slash 2)\cup(1\slash2, c_1)$. For $\vert \beta \vert>1$, the system has two attracting fixed points $c_0,c_1$ with basin of attraction $W^s(c_0)=[0,1\slash2)$, $W^s(c_1)=(1\slash2,1]$. The system has no $p$-periodic points for $p\geq 2$ since the function (\ref{fn:sigmoid.2}) is strictly monotone increasing on $[0,1]$. For $\beta=\pm1$, we apply theorem \ref{thm:nonhyperbolic}, to see that $x_*=1\slash 2$ is a asymptotically stable. \\
\\
The system undergoes a pitchfork bifurcation at $(x_*,\beta_*)=(1\slash 2,-1)$ and again at $(x_*,\beta_*)=(1\slash 2,1)$. For $\vert \beta\vert=1$ we have the following conditions on the map (\ref{fn:sigmoid.2}) and its derivatives at the fixed point $x_*=1\slash2$ 
\begin{eqnarray*}
\sigma(2\sigma(2x_*-1,\pm2\beta_*)-1,\pm2\beta_*)=1\slash 2 \\
\frac{\partial}{\partial x}\left[\sigma(2\sigma(2x-1,2\beta)-1,2\beta)\right]\mid_{x=x_*,\,\beta=\pm\beta_*}=1 \\
\frac{\partial}{\partial x^2}\left[\sigma(2x-1,2\beta)\right]\mid_{x=x_*,\,\beta=\pm\beta_*}=0\\
\frac{\partial}{\partial \beta}\left[\sigma(2\sigma(2x-1,2\beta)-1,2\beta)\right]\mid_{x=x_*,\,\beta=\pm\beta_*}=0 \\
\frac{\partial}{\partial \beta \partial x}\left[\sigma(2\sigma(2x-1,2\beta)-1, 2\beta)\right]\mid_{x=x_*,\,\beta=\pm\beta_*}\neq0\\
\frac{\partial}{\partial x^3}\left[\sigma(2\sigma(2x-1,2\beta)-1,2\beta)\right]\mid_{x=x_*,\,\beta=\pm\beta_*}\neq0. 
\end{eqnarray*}
Both bifurcations are super-critical. 
\end{proof}

\subsection{Parallel Updates}
The system of parallel updates is defined by the one-step map $F:\mathbb{R}^2\to \mathbb{R}^2$
\begin{eqnarray}
\begin{pmatrix} \zeta \\
\xi \end{pmatrix}\mapsto F(\zeta,\xi)=\begin{pmatrix}\sigma(2\xi-1,2\beta)\\ \sigma(2\zeta-1,2\beta).\end{pmatrix} \label{eqn:par.sys}
\end{eqnarray}
The dynamics of the parallel system are similar to the system studied in \cite{Blum.Wang.92}. As we shall show below, the parallel system exhibits periodic behavior that the sequential system does not and it follows as a corollary that the systems are not locally topologically conjugate. 

The parallelized CAVI algorithm is a dynamical system formed by iterations of $F$ defined in (\ref{eqn:par.sys}). We shall decouple the parallelized CAVI updates for sequences $\xi_k$ and $\zeta_k$ by looking at iterations of (\ref{fn:sigmoid.2}) acting on the sequences individually. This decoupling is visualized in diagram form  
\begin{eqnarray}
\begin{tikzcd}\label{diagram:decoupling}
   \zeta_0  \arrow[dr, ""]& \zeta_1  \arrow[dr,""]&  \zeta_2  \arrow[dr,""] & \zeta_3\cdots\\
  \xi_0 \arrow[ur, ""] & \xi_1  \arrow[ru, ""] & \xi_2 \arrow[ur, ""] & \xi_3\cdots
  \end{tikzcd}
\end{eqnarray}
where each cross is an application of $F$. The system formed the parallel updates is equivalent to the following decoupled systems of even subsequences and odd subsequences. The even subsequences are 
\begin{eqnarray} 
\zeta_{2k} &&= \sigma(2 \sigma(2\zeta_{2(k-1)}-1,2\beta)-1,2\beta),\quad k\geq 1 \label{seq:zeta.even}\\
\xi_{2k}&&=\sigma(2\sigma(2\xi_{2(k-1)}-1,2\beta)-1,2\beta), \quad k\geq 1. \label{seq:xi.even}
\end{eqnarray}
The odd subsequences are 
\begin{eqnarray} 
\zeta_{2k+1}&&=  \begin{cases} 
      \sigma(2\xi_{0},2\beta) & k=0\\
      \sigma(2\sigma(2\zeta_{2k-1},2\beta),2\beta) &k\geq 1
   \end{cases}
\label{seq:zeta.odd} \\
\xi_{2k+1}&&= \begin{cases} 
      \sigma(2\zeta_{0},2\beta) & k=0\\
      \sigma(2\sigma(2\xi_{2k-1},2\beta),2\beta) &k\geq 1.
   \end{cases} \label{seq:xi.odd}
\end{eqnarray}	
We investigate the dynamics of parallel systems by studying the dynamics of \eqref{fn:sigmoid.2} in the Appendix \ref{sec:dynamical} acting on each of the subsequences  individually.  Illustrations of the evolution of the dynamics of the parallel updates for various initializations and values of $\beta$ are in Figures \ref{fig:Par_CAVI_betaN12}-\ref{fig:Par_CAVI_beta12_2}. 
\begin{figure}[htbp!]
    \centering
    \includegraphics[width=10cm]{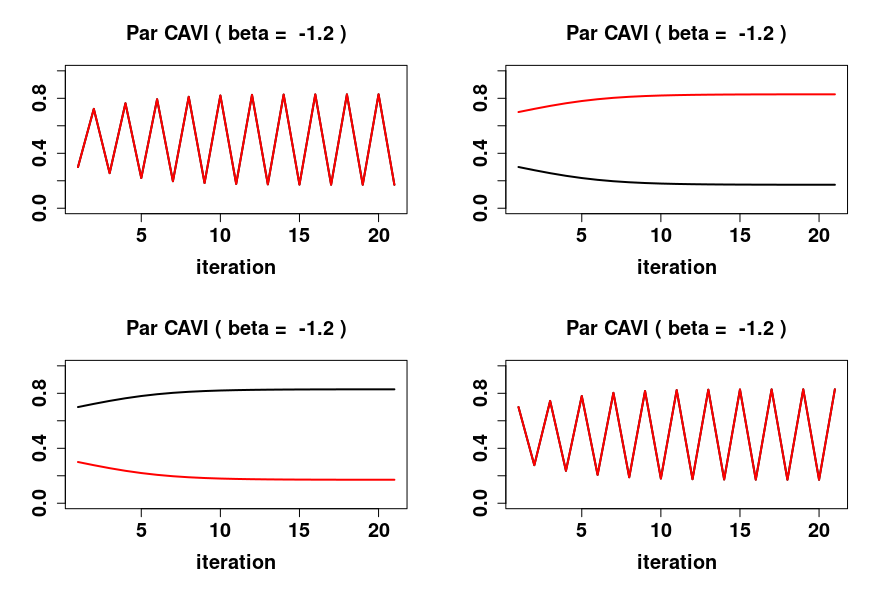}
    \caption{A plot of the first 20 iterations of the parallel update CAVI algorithm at various initializations for $\beta=-1.2$. In each of the plots the $\zeta$ updates are black and the $\xi$ updates are red. The upper left is an initialization of $\zeta_0=0.3$ and $\xi_0=0.7$, we see that this initialization converges to the two cycle $\mathcal{C}_0=\{(c_0,c_0),(c_1,c_1)\}$. The upper right plot is an initialization of $\zeta_0=0.3$ and $\xi_0=0.7$, we see that both of these converge to $c_0(1.2)\approx 0.17071$. The lower left plot is an initialization of $\zeta_0=0.7$ and $\xi_0=0.7$, we see that both of these converge to $c_1(1.2)\approx 0.82928$. The lower right is is an initialization of $\zeta_0=0.7$ and $\xi_0=0.3$, we see that this initialization converges to the two cycle $\mathcal{C}_0=\{(c_0,c_0),(c_1,c_1)\}$.}
    \label{fig:Par_CAVI_betaN12}
\end{figure}

\begin{figure}[htbp!]
    \centering
    \includegraphics[width=10cm]{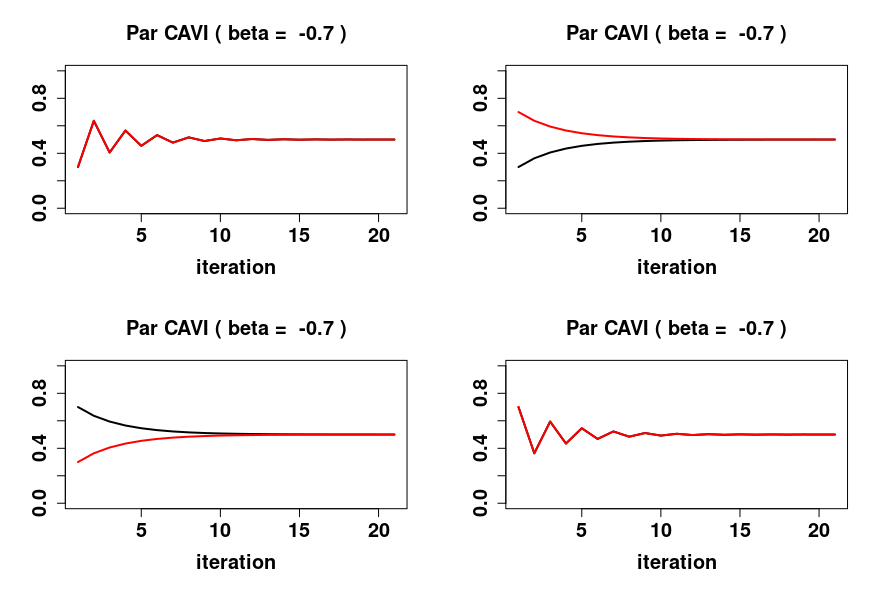}
    \caption{A plot of the first 20 iterations of the parallel update CAVI algorithm at various initializations for $\beta=-0.7$. In each of the plots the $\zeta$ updates are black and the $\xi$ updates are red. The upper left plot is an initialization of $\zeta_0=0.3$ and $\xi_0=0.3$, we see that both of these converge to the global fixed point $1\slash 2$. The upper right is an initialization of $\zeta_0=0.3$ and $\xi_0=0.7$, we see that this initialization converges to the global fixed point $1\slash 2$. The lower left is is an initialization of $\zeta_0=0.7$ and $\xi_0=0.3$, we see that this initialization converges to the global fixed point $1\slash 2$. The upper left plot is an initialization of $\zeta_0=0.7$ and $\xi_0=0.7$, we see that both of these converge to the global fixed point $1\slash 2$.}
    \label{fig:Par_CAVI_betaN07}
\end{figure}

\begin{figure}[htbp!]
    \centering
    \includegraphics[width=10cm]{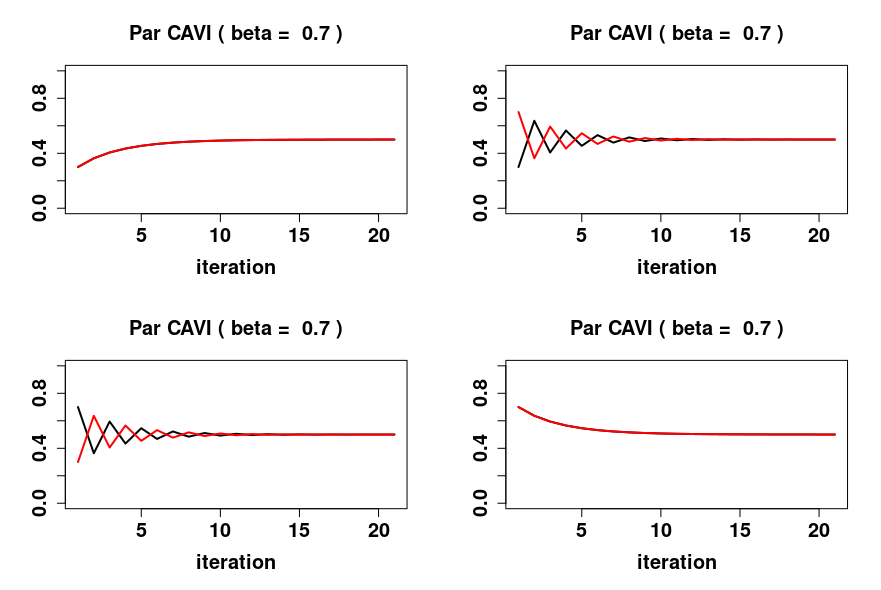}
    \caption{A plot of the first 20 iterations of the parallel update CAVI algorithm at various initializations for $\beta=0.7$. In each of the plots the $\zeta$ updates are black and the $\xi$ updates are red. The upper left plot is an initialization of $\zeta_0=0.3$ and $\xi_0=0.3$, we see that both of these converge to the global fixed point $1\slash 2$. The upper right is an initialization of $\zeta_0=0.3$ and $\xi_0=0.7$, we see that this initialization converges to the global fixed point $1\slash 2$. The lower left is is an initialization of $\zeta_0=0.7$ and $\xi_0=0.3$, we see that this initialization converges to the global fixed point $1\slash 2$. The upper left plot is an initialization of $\zeta_0=0.7$ and $\xi_0=0.7$, we see that both of these converge to the global fixed point $1\slash 2$. }
    \label{fig:Par_CAVI_beta07}
\end{figure}

\begin{figure}[htbp!]
    \centering
    \includegraphics[width=10cm]{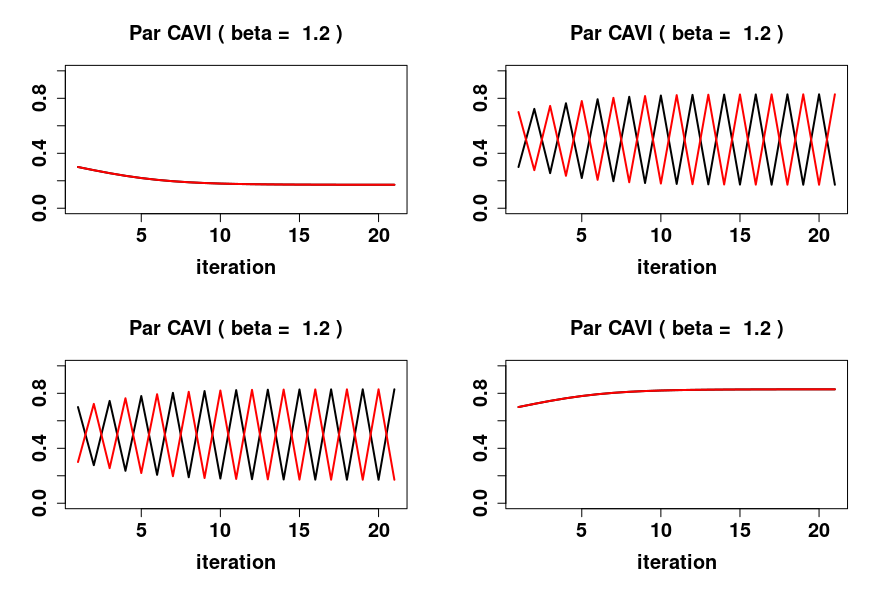}
    \caption{A plot of the first 20 iterations of the parallel update CAVI algorithm at various initializations for $\beta=1.2$. In each of the plots the $\zeta$ updates are black and the $\xi$ updates are red. The upper left plot is an initialization of $\zeta_0=0.3$ and $\xi_0=0.3$, we see that both of these converge to $c_0(1.2)\approx 0.17071$. The upper right is an initialization of $\zeta_0=0.3$ and $\xi_0=0.7$, we see that this initialization converges to the two cycle $\mathcal{C}_1=\{(c_1,c_0),(c_0,c_1)\}$. The lower left is is an initialization of $\zeta_0=0.7$ and $\xi_0=0.3$, we see that this initialization converges to the two cycle $\mathcal{C}_1=\{(c_1,c_0),(c_0,c_1)\}$. The lower right plot is an initialization of $\zeta_0=0.7$ and $\xi_0=0.7$, we see that both of these converge to $c_1(1.2)\approx 0.82928$.}
    \label{fig:Par_CAVI_beta12}
\end{figure}

\begin{figure}[htbp!]
    \centering
    \includegraphics[width=10cm]{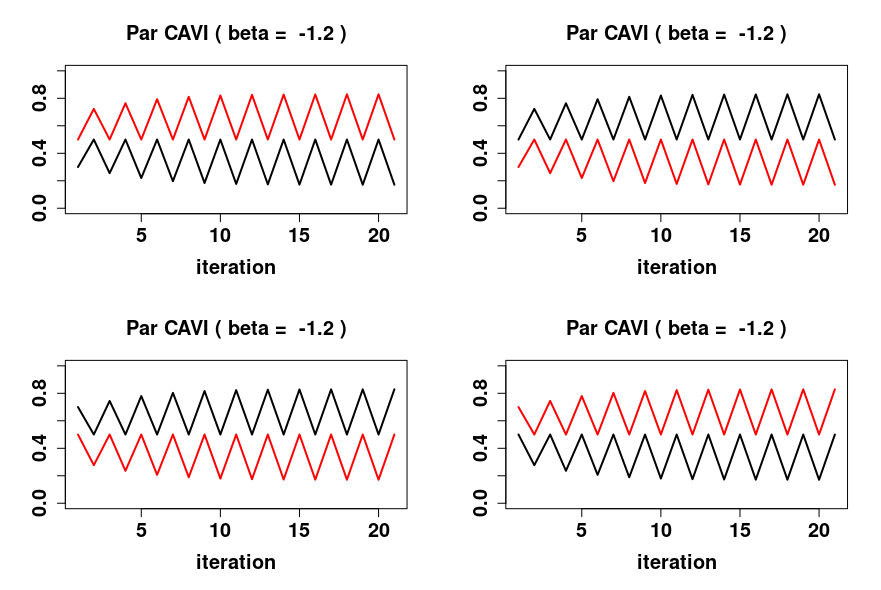}
    \caption{A plot of the first 20 iterations of the parallel update CAVI algorithm at various initializations for $\beta=-1.2$. In each of the plots the $\zeta$ updates are black and the $\xi$ updates are red. The upper left plot is an initialization of $\zeta_0=0.3$ and $\xi_0=0.5$, we see that this converges to the two-cycle $\mathcal{C}_2=\{(c_0,1\slash 2), (1\slash2, c_1)\}$. The upper right is an initialization of $\zeta_0=0.5$ and $\xi_0=0.3$, we see that this initialization converges to the two cycle $\mathcal{C}_3=\{(c_1,1\slash 2), (1\slash2, c_0)\}$. The lower left is is an initialization of $\zeta_0=0.7$ and $\xi_0=0.5$, we see that this initialization converges to the two cycle $\mathcal{C}_3$. The lower right plot is an initialization of $\zeta_0=0.5$ and $\xi_0=0.7$, we see that this converges to the two-cycle $\mathcal{C}_2$}
    \label{fig:Par_CAVI_beta12_1}
\end{figure}

\begin{figure}[htbp!]
    \centering
    \includegraphics[width=10cm]{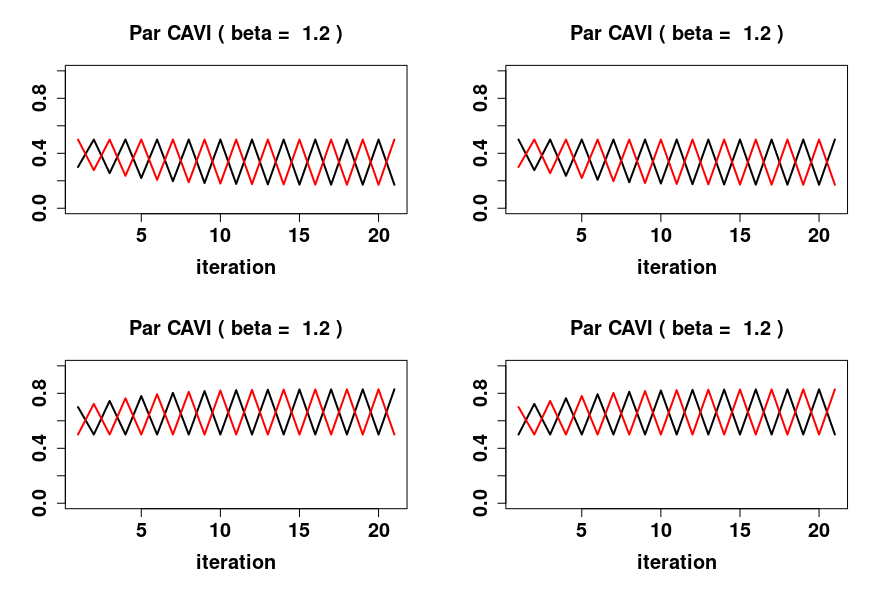}
    \caption{A plot of the first 20 iterations of the parallel update CAVI algorithm at various initializations for $\beta=1.2$. In each of the plots the $\zeta$ updates are black and the $\xi$ updates are red. The upper left plot is an initialization of $\zeta_0=0.3$ and $\xi_0=0.5$, we see that this converges to the two-cycle $\mathcal{C}_4=\{(c_0,1\slash 2), (1\slash2, c_0)\}$. The upper right is an initialization of $\zeta_0=0.5$ and $\xi_0=0.3$, we see that this initialization converges to the two cycle $\mathcal{C}_4$. The lower left is is an initialization of $\zeta_0=0.7$ and $\xi_0=0.5$, we see that this initialization converges to the two cycle $\mathcal{C}_5=\{(c_1,1\slash 2), (1\slash2, c_1)\}$. The lower right plot is an initialization of $\zeta_0=0.5$ and $\xi_0=0.7$, we see that this converges to the two-cycle $\mathcal{C}_5$}
    \label{fig:Par_CAVI_beta12_2}
\end{figure}

We now present the main result for the parallel dynamics. 
\begin{theorem}[Parallel Dynamics]\label{thm:parllel}
  	The dynamics of the parallel system (\ref{sys:par.cavi}) are as follows
  	\begin{enumerate}
  		\item For $\beta< -1$, the system has stable fixed points $(c_1,c_0)$ and $(c_0, c_1)$, and one unstable fixed point $(1\slash2, 1\slash 2)$, where $c_0$ and $c_1$ are the fixed points of (\ref{fn:sigmoid.2}). Furthermore the system exhibits periodic behavior in the form of 2-cycles. The  asymptotically stable 2-cycle, $\mathcal{C}_1=\{(c_0,c_0), (c_1,c_1)\}$ and asymptotically unstable 2-cycles, 
  		\[\mathcal{C}_2=\{(1\slash 2, c_1), (c_0, 1\slash 2)\} \text{ and } \mathcal{C}_3=\{(1\slash 2, c_0), (c_1, 1\slash 2)\}\].
  		The stable sets are
  		\begin{eqnarray*}
  		W^s(c_0,c_1)&&=[0,1\slash2)\times (1\slash2, 1] \\
  		W^s(c_1,c_0)&&= (1\slash2, 1] \times [0,1\slash2)\\
  		W^s(\mathcal{C}_1)&&= \left([0,1\slash2)\times [0, 1\slash 2)\right) \cup \left( (1\slash 2, 1]\times (1\slash 2,1] \right)\\
  		W^s(\mathcal{C}_2)&&= \left([0,1\slash2)\times \{ 1\slash 2\}\right) \cup \left( \{1\slash 2\}\times (1\slash 2,1] \right)\\
  		W^s(\mathcal{C}_3)&&= \left([0,1\slash2)\times \{ 1\slash 2\}\right) \cup \left( \{1\slash 2\}\times (1\slash 2,1] \right).
  		\end{eqnarray*}
  		\item For $-1\leq \beta\leq 1$, the system has a global attracting fixed point $(1\slash2, 1\slash 2)$. 
  		\item For $\beta>1$, the system has stable fixed points $(c_0,c_0)$ and $(c_1, c_1)$, and one unstable fixed point $1\slash2$, where $c_0$ and $c_1$ are the fixed points of (\ref{fn:sigmoid.2}). Furthermore the system exhibits periodic behavior in the form of 2-cycles. The  asymptotically stable 2-cycle, $\mathcal{C}_3=\{(c_0,c_0), (c_1,c_1)\}$ and asymptotically unstable 2-cycles, $\mathcal{C}_4=\{(1\slash 2, c_0), (c_1, 1\slash 2)\}$ and $\mathcal{C}_5=\{(1\slash 2, c_1), (c_1, 1\slash 2)\}$. The stable sets are
  		\begin{eqnarray*}
  		W^s(c_0,c_1)&&=[0,1\slash2)\times (1\slash2,1]\\
  		W^s(c_1,c_0)&&= (1\slash2, 1]\times [0,1\slash2)\\
  		W^s(\mathcal{C}_3)&&= \left([0,1\slash2)\times [0, 1\slash 2)\right) \cup \left( (1\slash 2, 1]\times (1\slash 2,1] \right)\\
  		W^s(\mathcal{C}_4)&&= \left([0,1\slash2)\times \{ 1\slash 2\}\right) \cup \left( \{1\slash 2, 1\}\times [0,1\slash 2) \right)\\
  		W^s(\mathcal{C}_5)&&= \left(\{1\slash2\} \times ( 1\slash 2,1]\right) \cup \left( (1\slash 2, 1]\times \{1\slash 2\} \right).
  		\end{eqnarray*}
  	\end{enumerate}
   The system has no $p$-periodic points for $p>2$. The system under goes a PD bifurcation at $\beta=-1$ and a pitchfork bifurcation at $\beta=1$.
\end{theorem}

\begin{proof}
The dynamics of the system defined by $F$ in (\ref{eqn:par.sys}) is equivalent to the dynamics of the system generated by the subsequences (\ref{seq:zeta.even}), (\ref{seq:zeta.odd}), (\ref{seq:xi.even}), (\ref{seq:xi.odd}). The dynamics of each of these subsequences is governed by the functions (\ref{fn:sigmoid.1}) and (\ref{fn:sigmoid.2}). By theorem \ref{thm:sigmoid.dynamics} we have the behavior for each of the subsequences (\ref{seq:zeta.even}), (\ref{seq:zeta.odd}), (\ref{seq:xi.even}), (\ref{seq:xi.odd}). For $\vert \beta\vert<1$, (\ref{fn:sigmoid.1}) has a globally stable fixed point at $x_*=1\slash 2$ and thus the only fixed point in the parallel system is $(1\slash 2,1\slash2)$ which must be globally stable. For $\beta=\pm 1$, the fixed point $x_0=1\slash 2$ is asymptotically stable by theorem \ref{thm:nonhyperbolic}. 

For $\beta>1$, (\ref{fn:sigmoid.1}) bifurcates. We have the unstable fixed point $x_*=1\slash 2$, as well as the two locally stable fixed points, $c_0$ with stable set $W^s(c_0)=[0,1\slash2)$, and $c_1$ with stable set $W^s(c_1)=(1\slash2,1]$. Returning to the system generated by $F$, if we consider the initialization $(\zeta_0,\xi_0)=(c_0,c_0)$ then by the sequence construction of $\zeta_n$, given in (\ref{seq:zeta.even}) and (\ref{seq:zeta.odd}), we see that $\zeta_{n}=c_0$ for $n\geq 1$, as $c_0$ is a fixed point of (\ref{fn:sigmoid.1}) for $\beta>1$. Similarly using the sequence construction of $\xi_n$, given in (\ref{seq:xi.even}) and (\ref{seq:xi.odd}), we see that $\xi_{n}=c_0$ for $n\geq 1$, as $c_0$ is a fixed point of (\ref{fn:sigmoid.1}) for $\beta>1$. Therefore $(c_0,c_0)$ is a fixed point. An analogous argument shows that $(c_1,c_1)$ is also a fixed point. The parallel system has the stable fixed points $(c_0,c_0)$ with stable set $W^s(c_0,c_0)=W^s(c_0)\times W^s(c_0)$ and $(c_1,c_1)$ with stable set $W^s(c_1,c_1)=W^s(c_1)\times W^s(c_1)$. After the bifurcation at $\beta=1$ the parallel system also contains 2-cycles. Using the sequence construction we see that $\mathcal{C}_3=\{(c_1,c_0),(c_0,c_1)\}$ is an asymptotically stable 2-cycle in the parallel system, with stable subspace $W^s(\mathcal{C}_3)=(1\slash2,1]\times [0,1\slash2) \cup [0,1\slash2)\times (1\slash2,1]$. Additionally we have two asymptotically unstable 2-cycles $\mathcal{C}_4=\{(c_0,1\slash 2),(1\slash 2,c_0)\}$ and $\mathcal{C}_5=\{(c_1,1\slash 2),(1\slash 2,c_1)\}$. Perturbing the $1\slash 2$ coordinate in the unstable cycle pushes it into the basin of attraction for one of the fixed points or the asymptotically stable 2-cycle. The stable sets are $W^s(\mathcal{C}_4)= \left([0,1\slash2) \times \{ 1\slash 2\}\right) \cup \left( \{1\slash 2, 1\}\times [0,1\slash 2) \right)$, $W^s(\mathcal{C}_5)= \left(\{1\slash2\} \times ( 1\slash 2,1]\right) \cup \left( (1\slash 2, 1]\times \{1\slash 2\} \right)$. The dynamics of $F$ lack any $p$-period point and cycles for $p>2$ as a consequence of its construction from (\ref{fn:sigmoid.2}). 

For $\beta<-1$, (\ref{fn:sigmoid.1}) bifurcates. We have the unstable fixed point $x_*=1\slash 2$, as well as the stable two cycle, $\mathcal{C}=\{c_0,c_1\}$ with stable set $W^s(\mathcal{C})=[0,1\slash2)\cup (1\slash 2, 1]$. Returning to the system generated by $F$, if we consider the initialization $(\zeta_0,\xi_0)=(c_0,c_1)$ then by the sequence construction of $\zeta_n$, given in (\ref{seq:zeta.even}) and (\ref{seq:zeta.odd}), we see that $\zeta_{n}=c_0$ for $n\geq 1$, as $\mathcal{C}$ is a 2-cycle of (\ref{fn:sigmoid.1}) for $\beta<-1$. Similarly using the sequence construction of $\xi_n$, given in (\ref{seq:xi.even}) and (\ref{seq:xi.odd}), we see that $\xi_{n}=c_1$ for $n\geq 1$, as $\mathcal{C}$ is a 2-cycle of (\ref{fn:sigmoid.1}) for $\beta<-1$. Therefore $(c_0,c_1)$ is a fixed point. An analogous argument shows that $(c_1,c_0)$ is also a fixed point. The parallel system has the stable fixed points $(c_0,c_1)$ with stable set $W^s(c_0,c_1)=W^s(c_0)\times W^s(c_1)$ and $(c_1,c_0)$ with stable set $W^s(c_1,c_0)=W^s(c_1)\times W^s(c_0)$, where $W^s(c_0)=[0,1\slash2)$ and $W^s(c_1)=(1\slash2,1]$. After the bifurcation at $\beta=1$ the parallel system also contains 2-cycles. Using the sequence construction we see that  $\mathcal{C}_1=\{(c_0,c_0),(c_1,c_1)\}$ is an asymptotically stable 2-cycle in the parallel system, with stable subspace $W^s(\mathcal{C}_1)= W^s(c_0)\times W^s(c_0) \cup W^s(c_1)\times W^s(c_1)$. Additional we have two asymptotically unstable 2-cycles $\mathcal{C}_2=\{(c_0,1\slash 2),(1\slash 2,c_1)\}$ and $\mathcal{C}_3=\{(c_1,1\slash 2),(1\slash 2,c_0)\}$. Perturbing the $1\slash 2$ coordinate in the unstable cycle pushes it into the basin of attraction for one of the fixed points or the asymptotically stable 2-cycle. The stable sets are $W^s(\mathcal{C}_3)= \left([0,1\slash2)\times [0, 1\slash 2)\right) \cup \left( (1\slash 2, 1]\times (1\slash 2,1] \right)$, $W^s(\mathcal{C}_4)= \left([0,1\slash2)\times \{ 1\slash 2\}\right) \cup \left( \{1\slash 2, 1\}\times [0,1\slash 2) \right)$, $W^s(\mathcal{C}_5)= \left(\{1\slash2\} \times ( 1\slash 2,1]\right) \cup \left( (1\slash 2, 1]\times \{1\slash 2\} \right)$. The dynamics of $F$ lack any $p$-period point and cycles for $p>2$ as a consequence of its construction from (\ref{fn:sigmoid.2}). 

This completes the characterization of the dynamics of $F$ for $\beta\in \mathbb{R}$. 
\end{proof}

\section{Edward--Sokal Coupling}
\label{sec:EScouple}
In this section we introduce the Edward--Sokal coupling following \cite{grimmett2006random}. We introduce a variational family for the Edward--Sokal coupling and derive the variational updates for this model. We show that the variational updates converge to a unique solution in a larger range than the equivalent Dobrushin regime for the corresponding Ising measure. 

\subsection{Random Cluster Model}
	Let $G=(V,E)$ be a finite graph. Let $e=\langle x,y\rangle \in E$ denote an edge in $G$ with endpoints $x,y\in V$.  Denote $\Sigma=\{1,2,\ldots,q\}^V$, $\Omega=\{0,1\}^E$, and $\mathcal{F}$ denote the powerset of $\Omega$. The random cluster model is a 2 parameter probability measure with an edge weight parameter $p\in[0,1]$ and a cluster weight parameter $q\in\{2,3,\ldots\}$ on $(\Omega, \mathcal{F})$ given by
	\begin{equation*}
	    \phi_{p,q}(\omega)\propto \left\{ \prod_{e\in E} p^{\omega(e)}(1-p)^{(1-\omega(e))}\right\}q^{\kappa(\omega)},
	\end{equation*}
	where $\kappa(\omega)$ denoted the number of connected components in the subgraph corresponding to $\omega$. 
	The partition function for the random cluster model is 
	\begin{equation*}
	   \mathcal{Z}_{RC}=\sum_{\omega\in\Omega} \left\{ \prod_{e\in E} p^{\omega(e)}(1-p)^{(1-\omega(e))}\right\}q^{\kappa(\omega)}. 
	\end{equation*}
	For $q=2$ the the random cluster model reduces to the Ising model on $G$. 
	
	The Edward--Sokal Coupling is a probability measure $\mu$ on $\Sigma\times\Omega$ given by
	\begin{equation}
	    \mu(\sigma,\omega)\propto \prod_{e\in E}\left[(1-p)\delta_{\omega(e),0}+p\delta_{\omega(e),1}\delta_{e}(\sigma)\right],
	\end{equation}
	where  $\delta_{a,b}=1(a=b)$, and $\delta_e(\sigma)=1(\sigma_{x}=\sigma_{y})$, for $e=(x,y)\in E$.

	It is well known that in the special case, $p=1-e^{-\beta}$ and $q=2$ the $\Sigma$-marginal of the ES coupling is the Ising model, the $\Omega$-marginal is the random cluster model \cite{grimmett2006random}. We are interested in better understanding how the convergence of the CAVI algorithm on the ES coupling compares to the convergence of the CAVI algorithm on the Ising model. 
	
\subsection{VI Objective Function}
	To calculate the VI updates for each variable we may need to make use of the alternative characterization of the ES coupling 
	$$\mu(\sigma,\omega)\propto \psi(\sigma)\phi_{p,1}(\omega)1_F(\sigma,\omega)$$
	where $\psi$ is uniform measure on $\Sigma$ and  $\phi_{p,1}(\omega)$ is a product measure on $\Omega$ 
	\begin{equation}
	    \phi_{p,1}(\omega)=\prod_{e\in E}p^{\omega(e)}(1-p)^{(1-\omega(e))}
	\end{equation}
	and 
	\begin{equation}
	    F=\left\{(\sigma,\omega): \delta_\omega(e)=1\implies \delta_e(\sigma)=1 \right\}
	\end{equation}
The variational family that we will be optimizing over is 
\begin{eqnarray}
q(\sigma,\omega)=q_1(\sigma_1)q_2(\sigma_2)q_0(\omega)1_{F}(\sigma,\omega). \label{ES.objective}
\end{eqnarray}
We have added the indicator on the set $F$ to eliminate the configurations $(\sigma\omega)$ that are not well defined in the variational objective. We will use the convention that $0\log(0)=0$. 
\subsection{VI updates}
The ELBO that corresponds to the variational family (\ref{ES.objective}) is 
\small 
\begin{eqnarray*}
\mbox{ELBO}(x_1,x_2,y,p)&=&x_1x_2y\log(x_1x_2y)-x_1x_2y\log(1-p)\\
&+&(1-x_1)x_2y\log((1-x_1)x_2y)-(1-x_1)x_2y\log(1-p)\\
&+& x_1(1-x_2)y\log(x_1(1-x_2)y)-x_1(1-x_2)y\log(1-p)\\
&+&(1-x_1)(1-x_2)y\log((1-x_1)(1-x_2)y)-(1-x_1)(1-x_2)y\log(1-p)\\
&+&x_1x_2(1-y)\log(x_1x_2(1-y))-x_1x_2(1-y)\log(p)\\
&+&(1-x_1)(1-x_2)(1-y)\log((1-x_1)(1-x_2)(1-y))-(1-x_1)(1-x_2)(1-y)\log(p). 
\end{eqnarray*}
\normalsize
Taking the derivative with respect to $x_1$ and simplifying gives us 
\small \begin{eqnarray*}
\mbox{ELBO}_1(x_1,x_2,y,p)&=&y\log\left(\frac{x_1}{1-x_1}\right)+(1-y)\log\left(\frac{1}{1-x_1}\right)\\
&+&x_2(1-y)\log(x_1(1-x_1))+
x_2(1-y)\log\left(\frac{x_2(1-x_2)(1-y)^2}{p^2}\right)\\
&+&\log\left(\frac{p}{(1-x_2)(1-y)}\right)+(2x_2-1)(1-y).
\end{eqnarray*}
\normalsize
Taking the derivative with respect to $x_2$ and simplifying gives us 
\small
\begin{eqnarray*}
\mbox{ELBO}_2(x_1,x_2,y,p)&=&y\log\left(\frac{x_2}{1-x_2}\right)+(1-y)\log\left(\frac{1}{1-x_2}\right)\\
&+&x_1(1-y)\log(x_2(1-x_2))+
x_1(1-y)\log\left(\frac{x_1(1-x_1)(1-y)^2}{p^2}\right)\\
&+&\log\left(\frac{p}{(1-x_1)(1-y)}\right)+(2x_1-1)(1-y). 
\end{eqnarray*}
\normalsize
Taking the derivative with respect to $y$ and simplifying gives us 
\small
\begin{eqnarray*}
\mbox{ELBO}_y(x_1,x_2,y,p)&=&x_1x_2\log\left(\frac{y}{1-y}\right)+x_1x_2\log\left(\frac{p}{1-p}\right)\\
&+& (1-x_1)(1-x_2)\log\left(\frac{y}{1-y}\right)+(1-x_1)(1-x_2)\log\left(\frac{p}{1-p}\right)\\
&+& (1-x_1)x_2\log\left(\frac{(1-x_1)x_2y}{1-p}\right)+x_1(1-x_2)\log\left(\frac{x_1(1-x_2)y}{1-p}\right)\\
&+&(1-x_1)x_2+x_1(1-x_2). 
\end{eqnarray*}
\normalsize
Absence of  closed form updates for any of the variables limits our ability to study the convergence of the system with classical dynamical systems techniques. Instead we look at the long evolution behavior of the system by plotting 100 iterations of the CAVI updates which are generated from the following system 
\begin{eqnarray*}
x_1(t+1)&=&\text{argmin}_{z\in (0,1)} \vert \mbox{ELBO}_1(z,x_2(t),y(t),p)\vert, \\
x_2(t+1)&=&\text{argmin}_{z\in (0,1)} \vert \mbox{ELBO}_2(x_1(t+1),z,y(t),p)\vert, \\
y(t+1)&=&\text{argmin}_{z\in (0,1)} \vert \mbox{ELBO}_y(x_1(t+1),x_2(t+1),z,p)\vert. 
\end{eqnarray*}
We generate the argmin of the free variable $z$ from a line search with a step size of $\Delta=10^{-6}$. Running these simulations we find that the iterations of $x_1(t),x_2(t),y(t)$ converge to a global solution within about $T=20$ time steps from any initialization $x_1(0),x_2(0),y(0)\in (0,1)$ and any $\beta>0$. It is evident that using the ES coupling, we get global convergence of the algorithm outside of the Dobrushin regime of the corresponding paramagnetic Ising model. The figures depicting the simulation results of convergence of the variational inference algorithm in the Edward--Sokal coupling can be found below in Figures \eqref{fig:ESp5iterations}-\eqref{fig:ESp01elbo}. 

\begin{figure}[htp!]
    \centering
    \includegraphics[width=10cm]{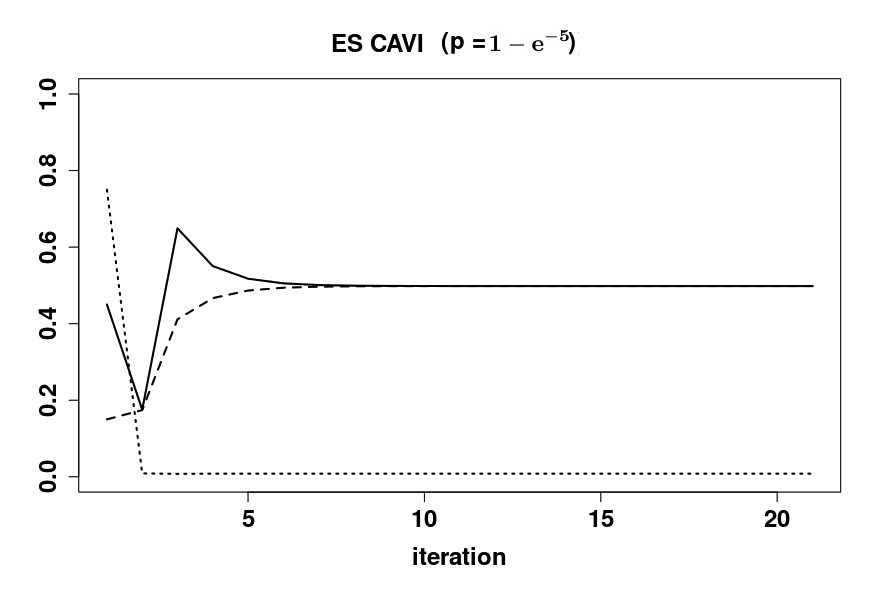}
    \caption{A plot of the 20 iterations of the ES updates for $p=1-e^{-5}$ from a uniformly random initialization. Each of the lines represents a different parameter. The solid line is $x_1$, the dashed line is $x_2$, and the dotted line is $y$. We see convergence to a unique fixed point for each of the variables.} 
    \label{fig:ESp5iterations}
\end{figure}

\begin{figure}[htp!]
    \centering
    \includegraphics[width=10cm]{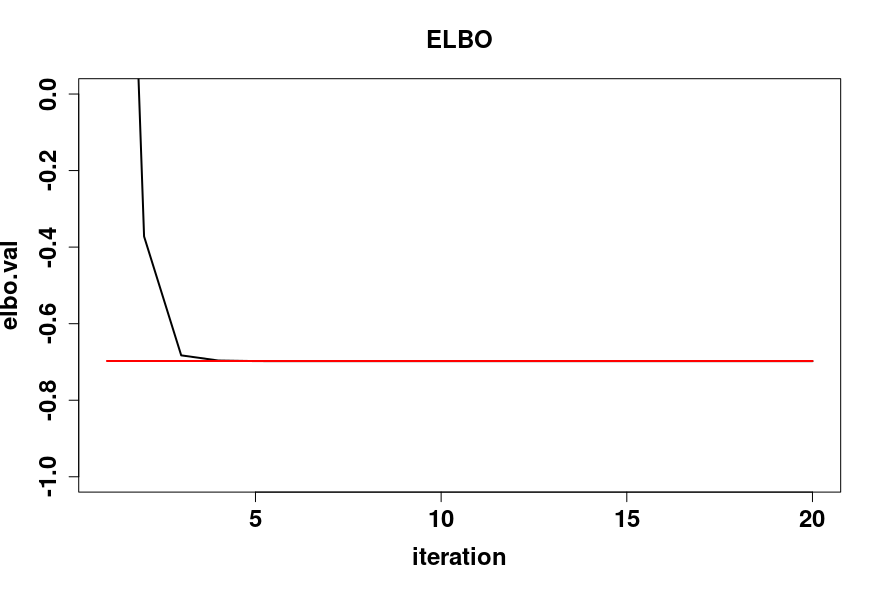}
    \caption{A plot of the ELBO of the ES coupling for $p=1-e^{-5}$. The red line denotes the global minimum ELBO value.}
    \label{fig:ESp5elbo}
\end{figure}

\begin{figure}[htp!]
    \centering
    \includegraphics[width=10cm]{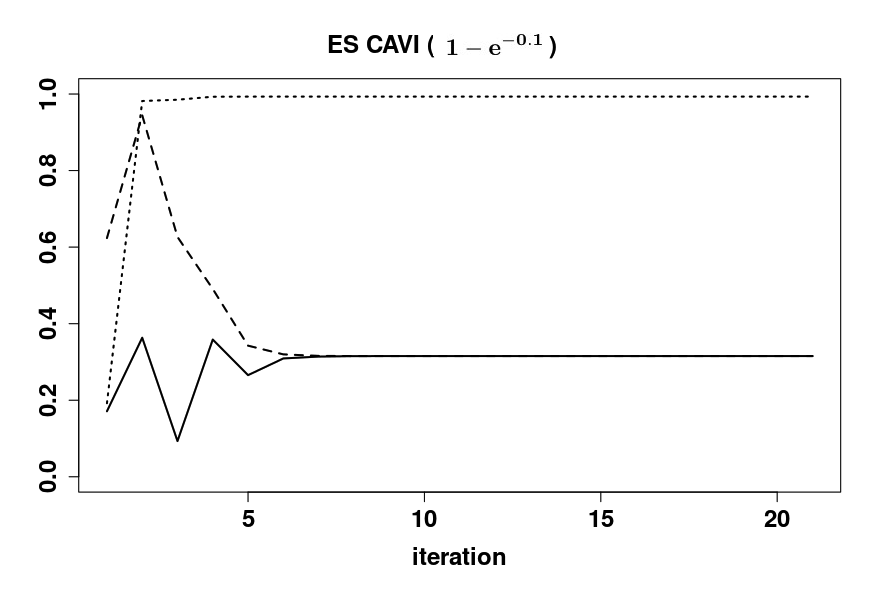}
    \caption{A plot of the 20 iterations of the ES updates for $p=1-e^{-0.1}$ from a uniformly random initialization. Each of the lines represents a different parameter. The solid line is $x_1$, the dashed line is $x_2$, and the dotted line is $y$. We see convergence to a unique fixed point.}
    \label{fig:ESp01iterations}
\end{figure}

\begin{figure}[htp!]
    \centering
    \includegraphics[width=10cm]{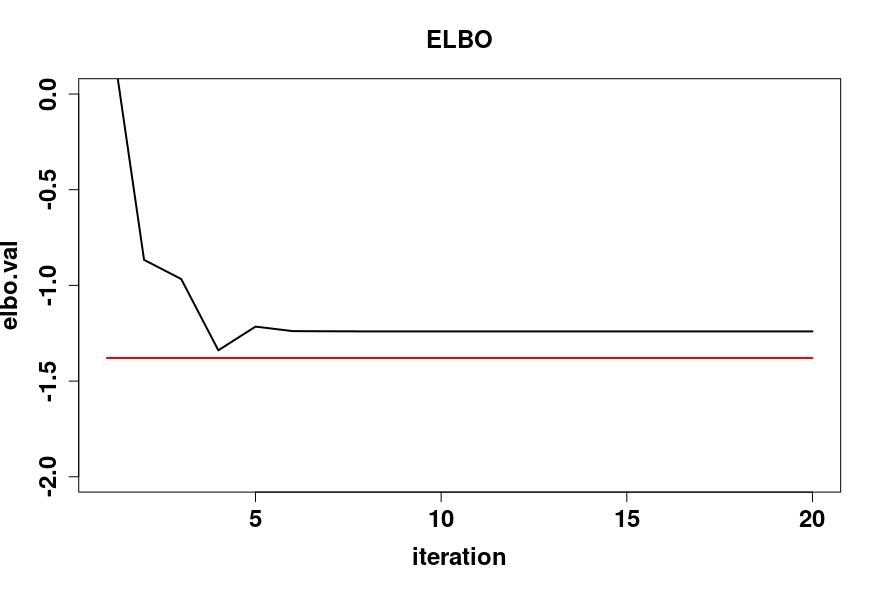}
    \caption{A plot of the ELBO of the ES coupling for $p=1-e^{-0.1}$. The red line denotes the global minimum ELBO value.}
    \label{fig:ESp01elbo}
\end{figure}

\newpage
\section{Conclusions}
\label{sec:conclusions} 
This paper demonstrates the use of classical dynamical systems and bifurcation theory to study the convergence properties of the CAVI algorithm of the Ising model on two nodes. In our simple model we are able to provide the complete dynamical behavior for the Ising model on two node. Interestingly we find that the sequential CAVI algorithm and parallelized CAVI algorithm are not topologically conjugate owing to the presence of periodic behavior in the parallelized CAVI. 

There are several open questions that stem from this work. What are types of dynamical behavior present in the Ising model CAVI in graphs with 3 or more nodes? What types of codimension 2 bifurcations occur in the CAVI algorithm for the Ising model? Extension to high dimensional models will require more sophisticated tools from dynamical systems. This is due to two simplifications that our arose in the above analysis. The Ising model on two nodes has the special property that both the sequential and parallel updates in the two variables case can be written as two separate one variable dynamical systems allowing for a simplified analysis. The analysis of systems in which the CAVI updates cannot be decoupled can be drastically complicated when it becomes necessary to compute center manifolds. This is the case with the Ising model on any graph with more than two nodes. While the theory needed to investigate these types of functions has long been established,  computational challenges needed to compute the center manifolds required to analyze the bifurcations of such systems remains. Software to handle such calculations has only recently been developed \cite{kuznetsov2019numerical}. A second complication arises when the codimension of the model is larger than two. Bifurcations of codimension 3 are so exotic that they are not well studied \cite{kuznetsov2008elements, kuznetsov2019numerical}. In practical terms this means that the convergence properties can only be studied numerically for models with a small number of parameters. Furthermore most of the numerical techniques work under the assumption of differentiability of the evolution operator and will fail to be applicable to many systems of practical interest in statistics such as the Edward--Sokal CAVI. 

An interesting question that stems from this work is the following: Can we study the convergence behavior of the CAVI algorithm in other systems to find similar parameter regimes that produce statistically optimal estimators? The answer to this question provides researchers with stable parameter regimes for the model. The non-existance of such a region would indicate the need for more expressive variational methods for the model beyond mean field methods.

Another avenue for future research lies in studying theoretical properties of other variational families. Most of the research into the theoretical properties of variational inference has focused on the mean field family due to its computational simplicity. This computational simplicity comes at the cost of limited expressive power. We are exploring convergence guarantees beyond mean field to accommodate more expressive variational families.

%

\appendix
%

\section{An overview of one dimensional dynamical systems}
\label{sec:dynamical}
The main focus of discrete dynamical systems is the asymptotic behavior of iterated systems \eqref{sequential.cavi}. Bifurcation theory studies how the dynamical behavior of a system changes as the parameter $J_{12}$  changes. We study the behavior of convergence of the CAVI algorithm by studying the autonomous discrete time dynamical system formed by the update equation \eqref{sequential.cavi}. This allows us to utilize tools from dynamical systems theory to study the behavior of the algorithm with respect to its parameters.
In this section we provide a brief overview of the necessary dynamical systems and bifurcation theory in dimension 1 used in section \ref{sec:main}. 

\subsection{Notation}
Our focus will be on parametric dynamical systems defined by a functions $f:\mathbb{R}^n\times \mathbb{R}^p\to \mathbb{R}^n$. We will call elements $\x\in\mathbb{R}^n$ elements in the state space (phase space) and elements $\alpha\in\mathbb{R}^p$ as parameters. We denote real numbers $x\in\mathbb{R}$ and real vectors in $\x =(x_1,\ldots,x_n) \in \mathbb{R}^n$ with bold. We denote the inverse of an invertible function $f$ by $f^{-1}$. The $k$-fold composition of a function $f$ with itself at a point $(\x,\alpha)$ will be denoted by $f^k(\x,\alpha)$. The $k$-fold composition of the inverse function $f^{-1}$ will be denoted $f^{-k}$. The identity function will be denoted $\mbox{id}$. We use the convention $f^0=\mbox{id}$. We denote the tensors of derivatives of $f$ by $f_\x(\x,\alpha)=(\partial f_i\slash \partial x_j)$, $f_{\x\x}(\x,\alpha)=(\partial f_i^2\slash \partial x_j \partial x_k)$,  $f_{\x\x}(\x,\alpha)=(\partial f_i^2\slash \partial x_j \partial x_k)$,  $f_{\x\x\x}(\x,\alpha)=(\partial f_i^2\slash \partial x_j \partial x_k \partial x_\ell)$, $f_\alpha(\x,\alpha)=(\partial f_i\slash \partial\alpha_j)$.

\subsection{Dynamical systems}
Dynamical systems is a classical approach to studying the convergence properties of non-linear iterative systems. These systems can be \textit{continuous in time}, for example a differential equation, or \textit{discrete in time}, for example iterations of a function from an initial point. A dynamical system is called \textit{autonomous} if the function governing the system is independent of time and \textit{non-autonomous} otherwise. The coordinate ascent variational inference for the Ising model is a discrete-time autonomous dynamical system. Before giving a complete proof of the dynamical properties of the CAVI algorithm for both the ferromagnetic and antiferromagnetic Ising model in dimension 2, we first give a basic introduction to the theory of discrete time dynamical systems and bifurcations following \cite{elaydi2007discrete, kuznetsov2008elements, kuznetsov2019numerical,wiggins2003introduction}.

Formally, a \textit{dynamical system} is tripe $\{T,X,\phi^t\}$ where $T$ is a time set, $X$ is the state space, and $\phi^t:X\to X$ is a family of evolution operators parameterized by $t\in T$ satisfying $\phi^0=id$ and $\phi^{s+t}=\phi^t\circ\phi^s$ for all $x\in X$. For a discrete time system the evolution operator is fully specified by the one-step map $\phi^1=f$, since the composition rule then defines $\phi^k=f^k$ for $k\in \mathbb{Z}$. We restrict the further discussion to discrete time dynamical systems defined by the one-step map 
\begin{eqnarray}
\x\mapsto f(\x,\alpha),\quad \x\in\mathbb{R}^n, \alpha\in\mathbb{R}^p, \label{dds}
\end{eqnarray}
where $f$ is a \textit{diffeomorphism}, a smooth function with smooth inverse, of the state space $\mathbb{R}^n$ and $\alpha$ are the parameters of the system. 

The basic geometric objects of a dynamical system are orbits in the state space and the \textit{phase portrait}, defined as follows. The \textit{phase portrait} is the partition of the state space induced by the orbits. The \textit{orbit} starting at a point $\x$ is an ordered subset of the state space $\mathbb{R}^n$ denoted $\mbox{orb}(\x)=\{f^k(\x):k\in \mathbb{Z}\}$. There are two special types of orbits, fixed points and cycles, defined below. 

A fixed point $\x_*$ of the system are points that remain fixed under the evolution of the system, ones that satisfies $\x_*=f(\x_*)$. We can classify fixed points of the system by studying the local behavior of the system near the fixed point. To do this we consider small perturbations of the system near the fixed point. A fixed point $\x_*$ is said to be \textit{locally stable} if points that are near the fixed point do not move to far away from the fixed point as the system evolves. Formally, if for any $\varepsilon >0$ there exists $\delta>0$ such that for all $x$ with $\vert \x-\x_*\vert <\delta$ we have $\vert f^k(\x)-\x_*\vert <\varepsilon$ for all $k>0$. A fixed point is called \textit{semi-stable from the right} if for any $\varepsilon >0$ there exists $\delta>0$ such that for all $x$ with $ 0< \x-\x_*<\delta$ we have $\vert f^k(\x)-\x_*\vert <\varepsilon$ for all $k>0$ (semi-stable from the left is defined analogously). It is said to be \textit{locally unstable} otherwise.  A fixed point $\x_*$ is \textit{locally attracting} if all points in a small neighborhood converge to the fixed point as we let the system evolve. Formally, if there exists an $\eta>0$ such that $\vert \x-\x_*\vert <\eta$ implies $f^n(\x)\to \x_*$ as $n\to \infty$. A fixed point $\x_*$ is \textit{locally asymptotically stable} if it is both locally stable and attracting. A fixed point $\x_*$ is \textit{locally semi-asymptotically stable from the right} if it is both locally semi-stable from the right and $\lim_n f^n(x)=x_*$ for $0< \x-\x_*<\eta$ for some $\eta$. It is \textit{globally asymptotically stable} if the point is attracting for all $\x$ in the state space. 

A \textit{cycle} is a periodic orbit of distinct points $C=\{\x_0,\x_1,\ldots, \x_{K-1}\}$, where $\x_0=f(\x_{K-1})$ for some $K>0$. The minimal $K$ generating the cycle is called the \textit{period} of the cycle. A subset $S\subset\mathbb{R}^n$ is called \textit{invariant} if $f^{k}(S)\subset S,\quad k\in \mathbb{Z}$. An invariant set $S$ is called \textit{asymptotically stable} if there exists a neighborhood $U$ of $S$ such that for any point in $U$ is eventually inside the set $S$. The \textit{stable set} of $S\subset \mathbb{R}^n$ is $W^s(S)=\left\{ \x\in \mathbb{R}^n :\lim_{k\to\infty}f^{k}(\x)\in S\right\}$. If $f$ is invertible, we define the \textit{unstable set} of $S\subset \mathbb{R}^n$ is $W^u(S)=\left\{ \x\in \mathbb{R}^n :\lim_{k\to \infty}f^{-k}(\x)\in S\right\}$. The unstable set of $S$ for the forward system $f^k$, $k>0$ is the stable set of $S$ for the backward system $f^{-k}$, $k>0$. It is possible to study the behavior of points that diverge by studying points that converge under the inverse map. We can also classify the stability of $K$-cycles. We classifying the stability of the cycle as a fixed point in the map $f^K$.

Consider a discrete time dynamical system defined by a diffeomorphism $f:\mathbb{R}\times\mathbb{R}\to \mathbb{R}$. Let $x_*$ be a fixed point of $f(x,\alpha)$ and consider a nearby point $x$, $\vert x-x_*\vert=\epsilon$. Taking a Taylor expansion of the system about the fixed point gives us 
\begin{eqnarray*}
f(x,\alpha)-x_* &=& f_x(x_*,\alpha)(x-x_*) + f_{xx}(x_*,\alpha)(x-x_*)^2 + O(\vert x-x_*\vert^3). 
\end{eqnarray*}
If the Jacobian does not have modulus one and $\epsilon$ is small enough, then the contribution by the terms of $O(\vert x-x_*\vert^2)$ will be negligible, in which case the behavior of the system is governed by the the behavior of the linearization of the system $f_x(x_*,\alpha)$. Motivated by this heuristic argument let us introduce several more terms. Assume that the Jacobian $A:=f_x(x_*,\alpha)$ of the system (\ref{dds}) at a fixed point $x_*$ is non-singular. Let $n_s$ denote the number of \textit{stable eigenvalues}, those with $\vert \lambda\vert <1$, $n_c$ denote the number of \textit{critical eigenvalues}, those with $\vert \lambda\vert =1$, and $n_u$ denote the number of \textit{unstable eigenvalues}, those with $\vert \lambda\vert >1$. Notice $n=n_s+n_c+n_u$. Let $E^s$, $E^u$, and $E^c$ denote the generalized invariant eigenspace of $A$ stable, unstable, and critical eigenvalues, respectively. A fixed point $x_*$ of the system (\ref{dds}) is called \textit{hyperbolic} if $A$ is non-singular and has no critical eigenvalues with $\vert\lambda\vert=1$ ($n_c=0$). A hyperbolic fixed point is called a \textit{hyperbolic saddle} if $A$ has both stable and unstable eigenvalues ($n_sn_u\neq 0$). A fixed point is \textit{non-hyperbolic} if $A$ has critical eigenvalues ($n_c>0$). The eigenvalues of a fixed point are called the \textit{multipliers} of the fixed point. The argument given above, the stability of a dynamical system is governed by the linearization of the system in a neighborhood of a hyperbolic fixed point, can be made rigorous. To do so, we need to discuss what it means for two dynamical systems to be equivalent. 

In the following, we define the notion of equivalence for dynamical systems. Two systems are topologically equivalent if we can map orbits of one system to orbits of another system in a continuous way that preserves the order of time. The dynamical system (\ref{dds}) is called \textit{topologically equivalent} to the system 
\begin{eqnarray}
\y \mapsto g(\y,\beta), \quad \y\in \mathbb{R}^n,\quad \beta\in \mathbb{R}^p, \label{dds2}
\end{eqnarray}
if there exists a homeomorphism of the parameter space $h_p:\mathbb{R}^p\to \mathbb{R}^p$, $\beta=h_p(\alpha)$ and a parameter dependent state space homeomorphism, continuous in the first argument, $h:\mathbb{R}^n\times \mathbb{R}^p \to \mathbb{R}^n$ such that, $\textbf{y}=h(\x,\alpha)$, mapping orbits of the system (\ref{dds}) at parameter value $\alpha$ onto orbits of the system (\ref{dds2}) at parameter $\beta=h_p(\alpha)$ preserving the direction of time. If $h$ is a diffeomorphism then the systems are called \textit{smoothly equivalent}. 

Let (\ref{dds}) and (\ref{dds2}) be two topologically equivalent invertible dynamical systems. Consider the orbit of the system under the mapping $f(\x,\alpha)$ , $\mbox{orb}(\x;f,\alpha)$, and the orbit of the system $g(\textbf{y},\beta)$, $\mbox{orb}(\y;g,\beta)$. Topological equivalence means that the homeomorphism $(h(\x,\alpha),h_p(\alpha) $ maps $\mbox{orb}(\x;f,\alpha)$ to $\mbox{orb}(\textbf{y};g,\beta)$ preserving the order of time. This gives us the following commutative diagram 
\begin{eqnarray*}
\begin{tikzcd}
    \cdots \arrow[d, "h"] \arrow[r,"f"]& f^{-1}(\x,\alpha) \arrow[d, "h"] \arrow[r,"f"]& \x \arrow[d, "h"] \arrow[r,"f"]&  f(\x,\alpha) \arrow[d,"h"] \arrow[r,"f"]&  \cdots \\
    \cdots \arrow[r,"g"]&  g^{-1}(\y,\beta) \arrow[r,"g"] & \y \arrow[r, "g"]  & g(\y,\beta) \arrow[r, "g"]& \cdots.
  \end{tikzcd}
\end{eqnarray*}
The orbits being topologically equivalent means that orbit $\x$ under the mapping $h$ should produce the same orbit as mapping $\x$ to $\y=h(\x,\alpha)$ computing the orbit of $\y$ under $g(\cdot,\beta)$ and mapping back to $f(\x,\alpha)$ by $h^{-1}$, $f(\x,\alpha)=h^{-1}\circ g\circ h(\x,\alpha)$. We shall primarily be interested in the behavior of the system in a small neighborhood of an equilibrium point. A system (\ref{dds}) is called \textit{locally topologically equivalent} near an equilibrium $\x_*$ to a system (\ref{dds2}) near an equilibrium $\y_*$ if there exists a homeomorphism $h:\mathbb{R}^n\to \mathbb{R}^n$ defined in a small neighborhood $U$ of $\x_*$ with $\y_*=h(\x_*)$ that maps orbits of (\ref{dds}) in $U$ onto orbits of (\ref{dds2}) in $V=h(U)$, preserving the direction of time.

We now have enough terminology to introduce the following theorem, which shows that the dynamics of a smooth system in the neighborhood of a hyperbolic fixed point are equivalent to the dynamics of the linearization of the system,
\begin{theorem}[Grobman-Hartman]\label{thm:GH}
    Consider a smooth map 
    \begin{eqnarray}
    \x\mapsto A\x+F(\x),\quad \x\in \mathbb{R}^n, \label{GH.eqn}
    \end{eqnarray}
    where $A$ is an $n\times n$ matrix and $F(\x)=O(\Vert \x\Vert^2)$. If $\x_*=0$ is a hyperbolic fixed point of (\ref{GH.eqn}), then (\ref{GH.eqn}) is topologically equivalent near this point to its linearization 
    \begin{eqnarray*}
    \x\mapsto A\x,\quad \x\in\mathbb{R}^n.
    \end{eqnarray*}
\end{theorem}
Note theorem \ref{thm:GH} is true for a general $n$-dimensional system. Theorem \ref{thm:GH} provides sufficient conditions to determine the stability of a hyperbolic fixed point of a general discrete time system,
 \begin{theorem}\label{thm:hyperbolic.stability}
     Consider a discrete time dynamical systems (\ref{dds}) where $f$ is a smooth map. Suppose for a fixed point $x_*$ that the eigenvalues of Jacobian $f_x(x_*,\alpha)$ all satisfy $\vert\lambda\vert<1$ then the fixed point is stable. Alternatively, suppose for a fixed point $x_*$ that the eigenvalues of Jacobian $f_x(x_*,\alpha)$ all satisfy $\vert\lambda\vert>1$ then the fixed point is unstable. 
 \end{theorem}
 
The linearization of the system near a non-hyperbolic fixed point is not sufficient to determine stability of the fixed point and we need to investigate higher order terms. The following theorem provides sufficient condition to check the stability of a smooth one dimensional system at a non-hyperbolic fixed point, 

\begin{theorem}\label{thm:nonhyperbolic}
Let $f:\mathbb{R}\times \mathbb{R} \to \mathbb{R}$. Suppose that $f(\cdot,\alpha)\in C^3(\mathbb{R};\mathbb{R})$ and $x_*$ is a non-hyperbolic fixed point of $f$, $x_*=f(x_*,\alpha)$. We have the following cases:\\
    Case 1: If $f_x(x_*,\alpha)=1$, then
\begin{enumerate}
    \item If $f_{xx}(x_*,\alpha)\neq 0$ then $x_*$ is semi-asymptotically stable from the left if $f_{xx}(x_*,\alpha),\alpha>0$ and semi-asymptotically stable from the right if $f_{xx}(x_*,\alpha)<0$
    \item if $f_{xx}(x_*,\alpha)=0$ and $f_{xxx}(x_*,\alpha)<0$ then $x_*$ is asymptotically stable
    \item if $f_{xx}(x_*,\alpha)=0$ and $f_{xxx}(x_*,\alpha)>0$ then $x_*$ is unstable
\end{enumerate}
    Case 2: If $f_x(x_*,\alpha)=-1$, then
    \begin{enumerate}
        \item If $\mathcal{S}f(x_*,\alpha)<0$, then $x_*$ is asymptotically stable
        \item If $\mathcal{S}f(x_*,\alpha)>0$, then $x_*$ is unstable. 
    \end{enumerate}
    where $\mathcal{S}f(x)$ is the \textit{Schwarzian derivative} of $f$
    \begin{eqnarray*}
    \mathcal{S}f(x,\alpha)=\frac{f_{xxx}(x,\alpha)}{f_{x}(x,\alpha)}-\frac{3}{2}\left[\frac{f_{xx}(x,\alpha)}{f_x(x,\alpha)} \right]^2.
    \end{eqnarray*}
\end{theorem}
The Schwarzian derivative controls the higher order behavior in oscillatory systems.

\subsection{Codimension 1 bifurcations}
Until now we have kept the parameter of the system fixed. The study of the change in behavior of a dynamical system as the parameters are varied is called bifurcation theory. A \textit{bifurcation} occurs when the dynamics of the system at a parameter value $\alpha_1$ differ from the dynamics of the system at a different parameter value $\alpha_2$. Changing the parameter in a system may cause a stable fixed point to become unstable, the fixed point may split into multiple fixed points, or a new orbit may form. Each of these is an example of a bifurcation, although these are not the only things that can happen. The point at which a bifurcation occurs is called a \textit{bifurcation point}. More formally, the parameter $\alpha_*$ is called a \textit{bifurcation point} if arbitrarily close to it there is $\alpha$ such that $\x\mapsto f(\x,\alpha), \x\in \mathbb{R}^n$ is not topologically equivalent to $\x\mapsto f(\x,\alpha_*), \x\in \mathbb{R}^n$ in some domain $U\subset\mathbb{R}^n$. 

A necessary, but not sufficient condition for bifurcation of a fixed point to occur is for the fixed point to be nonhyperbolic. Theorem \ref{thm:GH} together with the implicit function theorem show that in a sufficiently small neighborhood of a hyperbolic fixed point $(\x_*,\alpha_*)$, for each $\alpha$ there is another unique fixed point with the same stability properties as $(\x_*,\alpha)$. So hyperbolic fixed points do not undergo local bifurcations. In the context of discrete systems, a local bifurcation can occur only at a fixed point $(\x_*,\alpha_*)$ when the Jacobian of the system at $(\x_*,\alpha_*)$ has an eigenvalue with modulus one. 

Perhaps surprisingly, there are only three types of generic bifurcations that can happen in a discrete system with one parameter. They are the limit point (LP), period doubling (PD), and Neimark-Sacker (NS) bifurcations. The reason for this is fairly simple. It turns out that there is a generic system, called the \textit{topological normal form}, that undergoes this bifurcation at the origin in the $(\x,\alpha)$-plane. For any other system that undergoes the same bifurcation and satisfies certain non-degeneracy conditions there is a local change of coordinates that transforms the system into the topological normal form.

In general the types of bifurcations that can occur are connected to the number of parameters in the system. The minimal number of parameters that must be changed in order for a particular bifurcation to occur in $f(\x,\alpha)$ is called the \textit{codimension} of the bifurcation. A bifurcation is called \textit{local} if it can be detected in any small neighborhood of the fixed point, otherwise its called \textit{global}. Global bifurcations are much harder to analyze and since we do not attempt to investigate them in this paper we will not expand upon them further. More detailed results on bifurcations in codimension 1 and 2 can be found in \cite{kuznetsov2008elements, kuznetsov2019numerical}. 

We will now formally define the sufficient conditions for a system to undergo a period doubling or a pitchfork bifurcation. The \textit{period doubling bifurcation} occurs when a system with a non-hyperbolic fixed point with multiplier $\lambda_1=-1$ satisfies certain non-degeneracy conditions. There are two types of PD bifurcations. In the \textit{super-critical} case, a stable $2$-cycle is generated when a fixed point becomes unstable. In the \textit{sub-critical} case, a stable fixed point turns unstable when it coalesces with an unstable $2$-cycle \footnote{This is true for a general $k$-cycle. In the \textit{super-critical} case, a stable $2k$-cycle is generated when a $k$-cycle becomes unstable. In the \textit{sub-critical} case, a stable $k$-cycle turns unstable when it coalesces with an unstable $2k$-cycle }. The conditions for a PD bifurcation to occur are given as follows
\begin{theorem}[Period Doubling Bifurcation]\label{thm:pdb}
    Suppose that a one-dimensional system 
    \begin{eqnarray*}
    x\mapsto f(x,\alpha),\quad x,\alpha\in \mathbb{R},
    \end{eqnarray*}
    with smooth $f$, has at $\alpha=0$ the fixed point $x_*=0$, and let $\lambda=f_x(0,0)=-1$. Assume the following non-degeneracy conditions are satisfied 
    \begin{enumerate}
        \item $1\slash 2 (f_{xx}(0,0))^2 +1\slash 3f_{xxx}(0,0)\neq 0$
        \item $f_{x\alpha}(0,0)\neq 0$
    \end{enumerate}
    Then there are smooth invertible coordinate and parameter changes transforming the system into 
    \begin{eqnarray}
        \eta \mapsto-(1+\beta)\pm \eta^3+O(\eta^4).
    \end{eqnarray}
\end{theorem}

An classical example of a period doubling bifurcation can be seen in the logistic map $f(x,\mu)=\mu x(1-x)$, for $x\in[0,1]$. The bifurcation occurs at the point $(x_*,\mu_*)=(2\slash 3,3)$. The logistic map has two fixed points. One fixed point is at $x=0$ and the other is at $x=(\mu-1)\slash \mu$. We will ignore the fixed point at $x=0$ since it is repelling for $\mu>1$. We look at the behavior of the system in a small neighborhood of $\mu_*=3$. For $\mu=2.9$, the fixed point  $x_*=(\mu-1)\slash \mu$ is a hyperbolic attracting fixed point since $\vert f_x(x_*,2.9)\vert = \vert2-\mu\vert <1$. For $\mu=3$ the fixed point $x_*=(\mu-1)\slash \mu$ is a non-hyperbolic fixed point since $f_x(x_*,2.9)=2-\mu=-1$. Checking the Schwarzian derivative shows that the fixed point is asymptotically stable. For $\mu=3.1$, $x_*=(\mu-1)\slash \mu$ becomes a repelling fixed point. The points in $(0,x_*)\cup(x_*,1)$ converge to the attracting 2-cycle $C=\{0.558014, 0.7645665\}$. A super-critical period doubling bifurcation has occurred in the system formed by the logistic map. As the parameter $\mu$ increases we see a stable fixed point degenerate and a stable 2-cycle is formed. 

\begin{figure}[htbp]
    \centering
    \includegraphics[width=10cm]{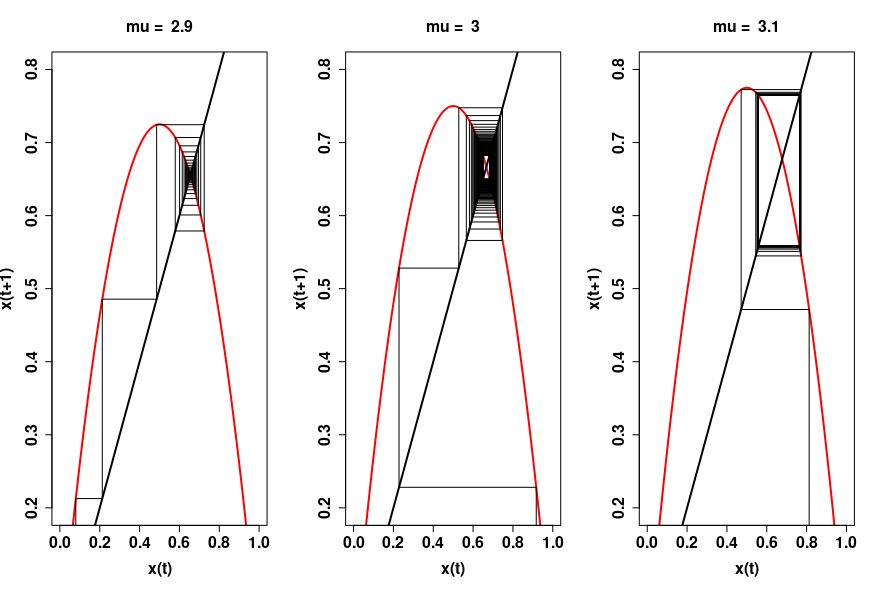}
    \caption{The above plots are cobweb diagrams for the logistic map $f(x,\mu)=\mu x(1-x)$, for $x\in[0,1]$, with parameters $\mu=2.9$, $\mu=3$ and $\mu=3.1$, respectively. For $\mu=2.9$ the system has one stable fixed point $x_*=(\mu-1)\slash \mu$. For $\mu=3$, the system has one non-hyperbolic fixed point $x_*=(\mu-1)\slash \mu$ which is asymptotically stable attracting; the plot was not iterated long enough to see convergence. For $\mu=3.1$, the system has a hyperbolic repelling fixed point $x_*=(\mu-1)\slash \mu$ and an asymptotically stable attracting two cycle $C=\{0.558014, 0.7645665\}$.}
    \label{fig:logistic.cobweb}
\end{figure}

The second iterate of a map that undergoes a PD bifurcation undergoes a bifurcation know as the \textit{pitchfork bifurcation}. A system that undergoes a \textit{super-critical} pitchfork bifurcation when a stable fixed point becomes unstable and two stable fixed points appear in the system. A system that undergoes a \textit{sub-critical} pitchfork bifurcation when two stable fixed points coalesce with an unstable fixed point, the unstable fixed point becomes stable as the parameter crosses the bifurcation point. Below we present extra details pertaining to the period doubling bifurcation and its relation to the pitchfork bifurcation. 

Consider the one-dimensional system 
\begin{eqnarray*}
x\mapsto -(1+\alpha)x+x^3=f(x,\alpha). 
\end{eqnarray*}
The map $f(x,\alpha)$ is invertible in a small neighborhood of $(0,0)$. The system has a fixed point at $x_*=0$ for all $\alpha$, with eigenvalue $-(1+\alpha)$. For small $\alpha<0$ the fixed point is hyperbolic stable and for $\alpha>0$ is it hyperbolic unstable. For $\alpha=0$ the fixed point is non-hyperbolic, but is asymptotically stable. 

Consider the second iterate of $f(x,\alpha)$
\begin{eqnarray*}
f^2(x,\alpha)&=&-(1+\alpha)f(x,\alpha)+(f(x,\alpha))^3\\
&=& (1+\alpha)^2x-\left[(1+\alpha)(2+2\alpha+\alpha^2)\right]x^3+O(x^5). 
\end{eqnarray*}
The second iterate has a trivial fixed point at $x_*=0$ and for $\alpha>0$ it has two non-trivial stable fixed points $x_{1}=(\sqrt{\alpha}+O(\alpha))$, $x_{1}=-(\sqrt{\alpha}+O(\alpha))$ that form a two cycle 
\begin{eqnarray*}
x_2=f(x_1,\alpha),\quad x_1=f(x_2,\alpha). 
\end{eqnarray*}
The conditions for a generic pitchfork bifurcation can be found in \cite{wiggins2003introduction}
\begin{theorem}[Pitchfork Bifurcation]\label{thm:pitchfork}
    For a system 
    \begin{eqnarray*}
    x\mapsto f(x,\alpha),\quad x,\alpha\in \mathbb{R}
    \end{eqnarray*}
    having non-hyperbolic fixed point at $x_*=0$, $\alpha_*=0$ with $f_x(0,0)=1$ undergoes a pitchfork bifurcation at $(x_*,\alpha_*)=(0,0)$ if 
    \begin{eqnarray*}
    f_\alpha(0,0)=0,\quad f_{xx}(0,0)=0, \quad f_{xxx}(0,0)\neq 0, \quad f_{x\alpha}(0,0)\neq 0. 
    \end{eqnarray*}
    A pitchfork bifurcation is \textit{super-critical} if $-f_{xxx}(x_*,\alpha_*)\slash f_{\alpha x}(x_*,\alpha_*)>0$ and \textit{sub-critical} if $-f_{xxx}(x_*,\alpha_*)\slash f_{\alpha x}(x_*,\alpha_*)<0$
\end{theorem}

An example of a pitchfork bifurcation can be seen in the second iteration of the logistic map $f^2(x,\mu)=\mu^2x(1-x)(1-\mu x(1-x))$, for $x\in[0,1]$. The bifurcation occurs at the point $(x_*,\mu_*)=(2\slash 3,3)$. For $\mu\leq 3$, the second iteration of the logistic map has the same fixed points as the first iteration. One fixed point is at $x=0$ and the other is at $x=(\mu-1)\slash \mu$. We will ignore the fixed point at $x=0$ since it is repelling for $\mu>1$. We look at the behavior of the system in a small neighborhood of $\mu_*=3$. For $\mu=2.9$, the fixed point  $x_*=(\mu-1)\slash \mu$ is a hyperbolic attracting fixed point since $\vert f^2_x(x_*,2.9)\vert <1$. For $\mu=3$ the fixed point $x_*=(\mu-1)\slash \mu$ is non-hyperbolic since $f^2_x(x_*,2.9)=2-\mu=1$. Checking the higher order derivative shows that the fixed point is asymptotically stable. For $\mu=3.1$, $x_*=(\mu-1)\slash \mu$ becomes a repelling fixed point. Using numerical methods we find two additional fixed points, $x_1=0.558014$  and $x_2=0.7645665$, both of which are attracting. A super-critical pitchfork bifurcation has occurred in the system formed by the logistic map. As the parameter $\mu$ increases we see a stable fixed point degenerates to an unstable fixed point and two stable fixed points. 

\begin{figure}[htbp]
    \centering
    \includegraphics[width=10cm]{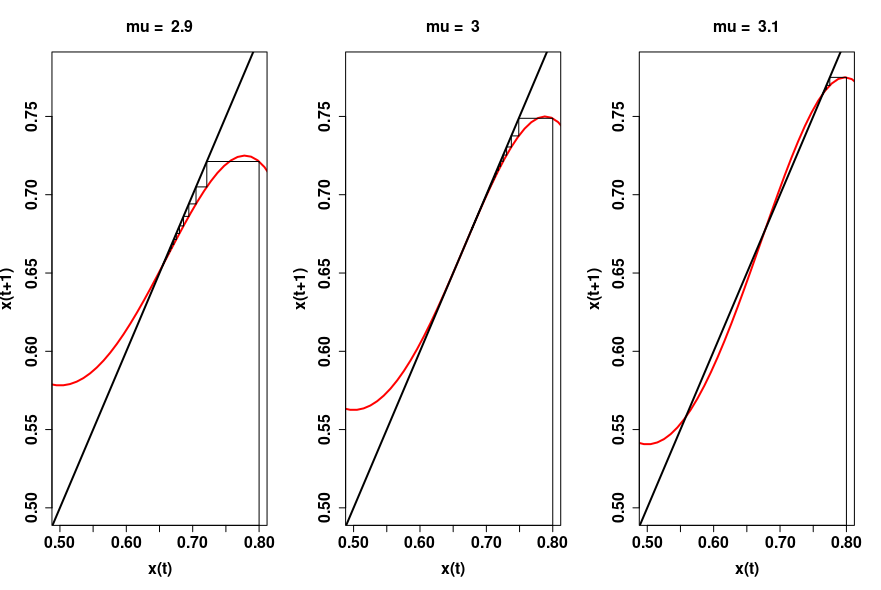}
    \caption{The above plots are cobweb diagrams for the second iterate of the logistic map $f(x,\mu)=\mu x(1-x)$, for $x\in[0,1]$, with parameters $\mu=2.9$ and $\mu=3.1$, respectively. For $\mu=2.9$ the system has one stable fixed point $x_*=(\mu-1)\slash \mu$. For $\mu=3.1$, the system has a hyperbolic repelling fixed point $x_*=(\mu-1)\slash \mu$ and two asymptotically stable attracting fixed points $x_1=0.0558014$ and $x_2=0.7645665$.}
    \label{fig:2logistic.cobweb}
\end{figure}

\bibliographystyle{plainnat}
\bibliography{references}


%
%
%
\end{document}